\documentclass{article}

\usepackage{arxiv}

\usepackage[title]{appendix}
\usepackage[utf8]{inputenc} 
\usepackage[T1]{fontenc}    
\usepackage[numbers]{natbib}
\usepackage{booktabs}       
\usepackage{amsfonts}       
\usepackage{amsmath,bm,upgreek}
\usepackage{hyperref}       
\usepackage{url}            
\usepackage{nicefrac}       
\usepackage{microtype}      
\usepackage{cleveref}       
\usepackage{graphicx}
\usepackage{doi}
\usepackage{subfig}
\usepackage{multirow}
\usepackage{booktabs}       
\usepackage{amsthm}      

\usepackage[title]{appendix}
\newtheorem*{theorem}{Theorem}

\graphicspath{ {./images/}}

\title{Derivative based global sensitivity analysis and its entropic link}

\date{January, 2025}

\author{ \href{https://orcid.org/0000-0001-8323-7406}{\includegraphics[scale=0.06]{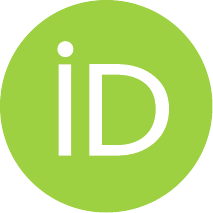}\hspace{1mm}Jiannan Yang}
    \\
  	School of Physics, Engineering and Technology\\
	University of York\\
	Heslington, York, YO10 5DD, UK\\
	\href{mailto:jiannan.yang@york.ac.uk}{jiannan.yang@york.ac.uk}\\
}

\renewcommand{\headeright}{Author Accepted Manuscript}
\renewcommand{\shorttitle}{\textit{derivative-based proxy for entropic GSA}}

\hypersetup{
pdftitle={},
pdfsubject={},
pdfauthor={Jiannan Yang},
pdfkeywords={},
}

\begin{document}
\maketitle

\begin{abstract}
Variance-based Sobol' sensitivity is one of the most well-known measures in global sensitivity analysis (GSA). However, uncertainties with certain distributions, such as highly skewed distributions or those with a heavy tail, cannot be adequately characterised using the second central moment only. Entropy-based GSA can consider the entire probability density function, but its application has been limited because it is difficult to estimate. Here we present a novel derivative-based upper bound for conditional entropies, to efficiently rank uncertain variables and to work as a proxy for entropy-based total effect indices. To overcome the non-desirable issue of negativity for differential entropies as sensitivity indices, we discuss an exponentiation of the total effect entropy and its proxy. Numerical verifications demonstrate that the upper bound is tight for monotonic functions and it provides the same input variable ranking as the entropy-based indices for about three-quarters of the 1000 random functions tested. We found that the new entropy proxy performs similarly to the variance-based proxies for a river flood physics model with 8 inputs of different distributions, and these two proxies are equivalent in the special case of linear functions with Gaussian inputs. We expect the new entropy proxy to increase the variable screening power of derivative-based GSA and to complement Sobol'-indices proxy for a more diverse type of distributions. 
\end{abstract}

\keywords{sensitivity proxy; sensitivity inequality; conditional entropy; exponential entropy; DGSM; Ishigami function}

\section{Introduction}
\label{sec:1}
This research is motivated by applications of global sensitivity analysis (GSA) towards mathematical models. The uncertain inputs of a mathematical model induce uncertainties in the output. GSA helps to identify the influential inputs and is becoming an integral part of mathematical modelling. 

The most common GSA approach examines variability using the output variance. Variance-based methods, also called Sobol' indices, decompose the function output into a linear combination of input and interaction of increasing dimensionality, and estimate the contribution of each input factor to the variance of the output \citep{sobol1993sensitivity,saltelli2008book}. As only the 2nd-order moments are considered, it was pointed out that the variance-based sensitivity measure is not well suited for heavy-tailed or multimodal distributions \citep{auder2008entropy,pianosi2015simple,liu2006relative}. Entropy is a measure of uncertainty similar to variance: higher entropy tends to indicate higher variance (for Gaussian, entropy is proportional to log variance). Nevertheless, entropy is moment-independent as it is based on the entire probability density function of the model output. It was shown in \cite{auder2008entropy} that entropy-based methods and variance-based methods can sometimes produce significantly different results. 

Both variance-based and entropy-based global sensitivity analysis (GSA) can provide quantitative contributions of each input variable to the output quantity of interest. However, the estimation of variance and entropy based sensitivity indices can become expensive in terms of the number of model evaluations. For example, the computational cost using sampling based estimation for variance-based total effect indices is $N(d+1)$ \citep{saltelli2008book, puy2020battle}, where $N$ is the base sample number and $d$ is the input dimension. Large values of $N$, normally in the order of thousands or tenths of thousands, are needed for more accurate estimate, and the computational cost has been noted as one of the main drawbacks of the variance-based GSA in \cite{saltelli2008book}. In addition, it was noted in \cite{auder2008entropy} that although both variance-based and entropy-based sensitivity analysis take long computational time, the convergence for entropy-based indices is even slower. 

In contrast, for a differentiable function, derivative-based methods can be more efficient. For example, Morris' method \citep{morris1991} constructs a global sensitivity measure by computing a weighted mean of the finite difference approximation to the partial derivatives, and it requires only a few model evaluations. The computational time required can be many orders of magnitude lower than that for estimation of Sobol’ sensitivity indices as demonstrated by \cite{kucherenko2009monte} and it is thus often used for screening a large number of input variables. 

Previous studies have found a link between the derivative-based measures and variance-based total effect indices. In \cite{campolongo2007effective}, a sensitivity measure  $\mu*$ is proposed based on the absolute values of the partial derivatives. It is empirically demonstrated that for some practical problems, $\mu*$ is similar to the variance-based total indices. In \cite{WAINWRIGHT2014gsa}, the variance-based sensitivity indices are interpreted as difference-based measures, where the total sensitivity index is equivalent to taking a difference in the output when perturbing one of the parameters with the other parameters fixed. The similarity to partial derivatives helps to explain why the mean of absolute elementary effects from Morris' method can be a good proxy for the total sensitivity index for detecting unessential variables. 

Observing the empirical success of the modified Morris' measure, \cite{kucherenko2009derivative} have proposed the so-called derivative-based global sensitivity measures (DGSM). This importance criterion is similar to the modified Morris' measure, except that the squared partial derivatives are used instead of their absolute values. In addition, an inequality link between variance-based global sensitivity indices and the DGSM is established in the case of independent Uniform or Gaussian input variables. This inequality between DGSM and variance-based GSA has been extended to input variables belonging to the large class of Boltzmann probability measures in \cite{lamboni2013sobol}. This link is via the Poincar\'{e} inequality and an optimal value for the scaling Poincar\'{e} constant has been developed in \cite{roustant2017poincare}. 

Inspired by the success of derivative-based proxies for Sobol' indices, in this paper, we present a novel derivative-based upper bound for entropy-based total effect indices. The key idea here is to make use of a well-known inequality between the entropy of a continuous random variable and its deterministic transformation. This inequality can be seen as a version of the information processing inequality and is shown here to provide an upper bound for the total effect entropy sensitivity measure. The entropy upper bounds are demonstrated to efficiently rank uncertain variables and can thus potentially be used as a proxy for entropy-based total effect indices for screening purposes. And that is the main contribution of this paper. 

In addition, via exponentiation, we extend the upper bound to the widely used DGSM for a total sensitivity measure based on entropy power (also known as effective variance). In the special case with Gaussian inputs and linear functions, the proposed new proxy for entropy GSA is found to be equivalent to the proxy for variance-based total effect sensitivity. Furthermore, unlike the variance proxy, the inequality link between derivatives and entropy does not require the random inputs to be independent. The new entropy proxy is thus expected to not only increase the variable screening power of derivative-based GSA, but can also complement variance proxies for a more diverse type of distributions.

Note that we focus on the standard derivative-based sensitivity measures in this paper. A closely related derivative-based sensitivity analysis technique is active subspace, which makes use of the leading eigenspaces of the second moment matrix of the function partial derivatives. DGSM indices are the diagonal of this second moment matrix. Sensitivity indices based on active subspace have been found to bound DGSM and Sobol' total effect indices for scalar-valued outputs \citep{constantine2017global}, and a generalization to vector-valued functions with Gaussian inputs has been discussed by \cite{zahm2020gradient}. Instead of the function derivatives, active sensitivity directions can also be obtained from the derivatives of the output distributions using the leading eigenvectors the Fisher Information Matrix (FIM) \citep{Yang_SAframework_2022}. Use of the leading symplectic eigenspace of the FIM has also been proposed for decision oriented sensitivity analysis \citep{Yang_Symplectic_2023}. 

It should also be noted that, in addition to entropy-based measures, there are many other moment-independent sensitivity measures, such as the $\delta$-indicator \citep{BORGONOVO_deltaMeasure_2007}, maximum mean discrepancy \citep{da2021kernel}, Kolmogorov–Smirnov statistic \citep{pianosi2015simple} and Kullback–Leibler (KL) divergence \citep{krzykacz2001epistemic,liu2006relative}. More discussions can be found in \citep{BORGONOVO_review_2016}. 

Many of the above mentioned measures can be seen as special cases of the Csiszar f-divergence between the conditional and unconditional output densities \citep{BORGONOVO_review_2016, da2015global}. The expected f-divergence can be reformulated in terms of the joint density of the input and output distributions and the product of their marginals, and this allows the application of general dependence measures to sensitivity analysis, such as the Distance Correlation based on characteristic functions and the Hilbert–Schmidt Independence Criterion (HSIC) that generalizes the notion of covariance between two random variables \citep{da2015global}.

In what follows, we will first review global sensitivity measures in Section \ref{sec:2}, where the motivation for the entropy-based measure and its proxy is discussed with an example. In Section \ref{sec:3}, we first establish the inequality relationship between the total effect entropy measure and the partial derivatives of the function of interest, where mathematical proofs are provided. An exponential version of the inequality link is further proposed in Section \ref{sec:3}. In Section \ref{sec:4}, we provide numerical assessments with analytical functions. In Section \ref{sec:5}, a river flood physics model is used to demonstrate the effectiveness of the new entropy proxy. Concluding remarks are given in Section \ref{sec:6}. 
\section{Global sensitivity measures}
\label{sec:2}
In our context, an important question for GSA is: `Which model inputs can be fixed anywhere over its range of variability without affecting the output?' \citep{saltelli2008book}.  In this section, we first review both variance-based and derivative-based sensitivity measures, which can provide answers to the above `screening' question. We then use a simple example to motivate the use of entropy-based sensitivity indices. 
\subsection{Total variance effect and its link with derivative-based GSA}
\label{sec:2.1}
The variance-based total effect measure accounts for the total contribution of an input to the output variation, and is often a preferred approach due to its intuitive interpretation and quantitative nature.

Let us denote $\mathbf{x} = (x_1, x_2, \dots, x_d)$ as independent random input variables, and $y$ being the output of our computational model represented by a function $g$, such that $y = g(\mathbf{x}) $. 

The variance-based GSA decompose the output variance $V(Y)$ into conditional terms \citep{Hoeffding1948,sobol1993sensitivity}:
\begin{equation} \label{eq:Vy}
   V(Y) = \sum_i{V_i} + \sum_i \sum_{i < j} V_{ij} + \cdots + V_{1,2,\dots,d}
\end{equation}
where
\begin{equation*} 
   V_i = V[\mathbb{E} (Y|X_i)] ; \quad V_{ij} = V[\mathbb{E} (Y|X_i, X_j)] - V[\mathbb{E} (Y|X_i)] - V[\mathbb{E} (Y|X_j)]  
\end{equation*}
and so on for the higher interaction terms. $V_i$ measures the first order effect variance and $V_{ij}$ for a second order effect variance, where their contributions to the unconditional model output variance can be quantified as $V_i/V(Y)$ and $V_{ij}/V(Y)$ respectively. Analogous formulas can be written for higher-order terms, enabling the analyst to quantify the higher-order interactions.

The total order sensitivity index is then defined as:
\begin{equation} \label{eq:Ti}
    S_{T_i} = \frac{\mathbb{E} [V(Y|\mathbf{X}_{\sim i})]}{V(Y)} = \frac{ V_{T_i}}{V(Y)}
\end{equation}
where $\mathbf{X}_{\sim i}$ is the set of all inputs except $X_i$, and $\mathbb{E} [V(Y|\mathbf{X}_{\sim i})] = V(Y) - V[\mathbb{E} (Y|\mathbf{X}_{\sim i})]$ is the remaining variance if the true values of $\mathbf{X}_{\sim i}$ can be determined. The total order sensitivity index measures the total contribution of the input $X_i$ to the output variance, including its first order effect and its interactions of any order with other inputs.

When the function $g(\cdot)$ is differentiable, local sensitivity can be measured using the square integrable partial derivatives $g_{x_i}=\partial g/\partial x_i$ which can be seen as a limiting version of Morris' elementary effect when the incremental step tends to zeros \citep{kucherenko2009derivative}. The partial derivative depends on a nominal point. For global sensitivity analysis, an average of the partial derivatives can be taken over the input parameter space: 
\begin{equation} \label{eq:Deri.Morris}
   \mu_i = \mathbb{E} \left[ \left| \frac{\partial g (\mathbf{X})}{\partial x_i} \right| \right] =\int \left| \frac{\partial g (\mathbf{x})}{\partial x_i} \right| f_X (\mathbf{x}) d \mathbf{x}
\end{equation}
for $i = 1, 2, \dots, d$ and $f_{X}(\mathbf{x})$ is the joint probability density function (PDF) of $\mathbf{x}$. As pointed out by \cite{kucherenko2009derivative}, for uniformly distributed inputs, the measure $\mu_i$ can be seen as a limiting version of the modified Morris' index $\mu^*$. 

Based on that observation, the Derivative-based Global Sensitivity Measure (DGSM): 
\begin{equation} \label{eq:Deri.DGSM}
   \nu_i = \mathbb{E} \left[ \left(  \frac{\partial g (\mathbf{X})}{\partial x_i} \right)^2 \right] = \int \left(  \frac{\partial g (\mathbf{x})}{\partial x_i} \right)^2 f_X (\mathbf{x}) d \mathbf{x}
\end{equation}
has been proposed to be used as a proxy for $S_{T_i}$ \citep{kucherenko2009derivative} to detect un-influential input variables. In particular, the total sensitivity variance $V_{T_i} = \mathbb{E} [V(Y|\mathbf{X}_{\sim i})]$ is upper bounded by DGSM via the following inequality: 
\begin{equation} \label{eq:STi_DGSM}   
   V_{T_i} \leq C_i \nu_i
\end{equation}
based on Poincar\'{e} inequality and the Poincar\'{e} constants were found to be optimal for $C_i$ \citep{lamboni2013sobol, roustant2017poincare}. Note that independent input variables are required for the DGSM-based upper bound as it is based on variance decomposition. 

Eq \ref{eq:STi_DGSM} thus provides a screening method using the upper bound, which is typically computationally faster compared to a direct estimation of the Sobol' indices. Tighter the upper bound, more effective the low cost screening is. 

Divide Eq \ref{eq:STi_DGSM} by the output variance $V(Y)$ from both sides, we then have the upper bound $S_{T_i} \leq C_i \nu_i /V(Y)$ which can be used as a proxy for $S_{T_i}$ for variable screening. For Gaussian inputs with variance $\sigma_i^2$, $C_i = \sigma_i^2 $ and the inequality in Eq \ref{eq:STi_DGSM} becomes $S_{T_i} \leq \sigma_i^2 \nu_i / V(Y)$ as given in \cite{kucherenko2009derivative}.
\subsection{A motivating example and entropy-based sensitivity}
\label{sec:2.2}
\cite{pianosi2015simple} pointed out that a major limitation of variance-based sensitivity indices is that they implicitly assume that output variance is a sensible measure of the output uncertainty. However, if the output distribution is multi-modal or if it is highly skewed, using variance as a proxy of uncertainty may lead to contradictory results. 

To illustrate this point, we look at the simple function $y=x_1/x_2$. In this case, the two inputs both follow the chi-squared $\chi^2$ distribution with $x_1 \sim \chi^2(10) $ and $x_2 \sim \chi^2(13.978)$, and are assumed to be independent. This results in a positively skewed distribution of $Y$ with a heavy tail. This example has been used by \cite{liu2006relative} to demonstrate the limitation of variance-based sensitivity indices, where they propose a Kullback-Leibler (KL) divergence based metric:
\begin{equation} \label{eq:KL}   
   KL_{T_i} = \int f_1(y(x_1,\dots,\bar{x}_i,\dots,x_d)) \ln \frac{f_1(y(x_1,\dots,\bar{x}_i,\dots,x_d))}{f_0(y(x_1,\dots,x_i,\dots,x_d))} dy
\end{equation}
In Eq \ref{eq:KL}, $f_1(y)$ and $f_0(y)$ are the PDFs of the output, depending on whether $x_i$ is fixed, usually at its mean. The larger the $KL_{T_i}$, the more important $X_i$ is. It was found in \cite{liu2006relative} that the effect of $X_1$ is higher in terms of divergence of the output distribution, but the variance-based total index shows that $X_1$ and $X_2$ are equally important. The higher influence of $X_1$ has also been confirmed in \cite{pianosi2015simple} where the sensitivity is characterised by the change of cumulative distribution function of the output. 
\begingroup
\def\arraystretch{0.5}
\setlength{\tabcolsep}{4pt}
\begin{table}
	\caption{Sensitivity results for $y = x_1/x_2$. }
    \centering
\makebox[\textwidth][c] {
	\begin{tabular}{cccc} 	
		\toprule
		             & \multirow{2}{3cm}{ \centering Variance-based} & \multirow{2}{4cm}{ \centering K-L divergence based } &   \multirow{2}{4cm}{\centering Entropy-based}\\  
		               & &  \\
		         Variable &  $S_{T_i}$&  $KL_{T_i}$  & $\eta_{T_i}$ \\  
		\cmidrule(lr){1-4}   
	              $x_1 \sim \chi^2(10)$ &  0.546  & 0.1571  & 0.510  \\
	              $x_2 \sim \chi^2(13.978)$ &  0.547 & 0.0791  & 0.213 \\ 
		\bottomrule
	\end{tabular}
	}
	\label{tab:x1/x2}
\end{table}
\endgroup

We reproduce the sensitivity results of $S_{T_i}$ and $KL_{T_i}$ in Table \ref{tab:x1/x2} from \cite{liu2006relative}. In addition, we also compare the results with the entropy-based total sensitivity index (ETSI) \citep{kala2021global}:  
\begin{equation} \label{eq:Entropy.total}
   \eta_{T_i} =  \frac{\mathbb{E} [H(Y|\mathbf{X}_{\sim i} )]}{H(Y)} = \frac{H_{T_i}}{H_Y}
\end{equation}
which measures the remaining entropy of $Y$ if the true values of $\mathbf{X}_{\sim i}$ can be determined, in analogy to the variance-based total effect index $S_{T_i}$. $H$ is the differential entropy, that is:
\begin{equation*} \label{eq:Entropy.def}
   H_Y=H(Y)= - \int f(y) \ln f(y) dy
\end{equation*}
and the conditional differential entropy is defined accordingly as:
\begin{equation*} \label{eq:Entropy.conditional}
   \mathbb{E} [H(Y|\mathbf{X}_{\sim i} )]  = - \int f(y, \mathbf{x}_{\sim i}) \ln f(y| \mathbf{x}_{\sim i}) dy d{\mathbf{x}_{\sim i}}
\end{equation*}
where the integral is with respect to the support set of the random variables and where $\sim i$ indicates the index ranges from $1$ to $d$ excluding $i$. The conditional PDF can in general be written as $f(y| \mathbf{x}_{\sim i})= f(y, \mathbf{x}_{\sim i})/ f(y)$, except for cases where the differential entropy becomes infinite. 

The total effect entropy $H_{T_i}$ is a global measure of uncertainty as the expectation is with respect to all possible values of $\mathbf{X}_{\sim i}$. $\eta_{T_i} \leq 1$ because $ H_{T_i} ={\mathbb{E} [H(Y|\mathbf{X}_{\sim i} )]} \leq {H(Y)}$ with equality if and only if $\mathbf{X}_{\sim i} $ and $Y$ are independent  \citep{cover1999elements}. However, $\eta_{T_i}$ can be negative as it is defined using differential entropy. This is undesirable as sensitivity indices and later in Section \ref{sec:3} we propose an exponentiation to overcome this issue.  

$\eta_{T_i}$ has been estimated numerically based on the histogram method given in Appendix \ref{appendix:a}, where $10^7$ samples are used. Table \ref{tab:x1/x2} shows that the entropy-based $\eta_{T_i}$ is able to effectively identify the higher influence of $X_1$. Note that different from $KL_{T_i}$ which is conditional on the value of $x_i$($x_i$ are set at their mean values in Table \ref{tab:x1/x2}), $\eta_{T_i}$ is an un-conditional sensitivity measure as all possible values of the inputs are averaged out. 

We note in passing that analogously to the variance based sensitivity indices, a first order entropy index can also be defined as $\eta_i = (H(Y) - \mathbb{E}[H(Y|X_i)])/{H(Y)} = I(X_i,Y)/{H(Y)}$ \citep{krzykacz2001epistemic}. $I(X_i,Y)$ is the mutual information which measures how much knowing $X_i$ reduces uncertainty of $Y$ or vice versa. The index $\eta_i$ can thus be regarded as a measure of the excepted reduction in the entropy of of $Y$ by fixing $X_i$. 
\subsection{Summary}
\label{sec:2.4}
From the motivating example, it became clear that $S_{T_i}$ might not be very indicative for variable rankings with outputs of general distribution shapes. This is especially the case for highly skewed or multi-modal distributions. This limitation is overcome by the entropy-based measures which are applicable independent of the underlying shape of the distribution. 

However, the entropy-based ETSI have limited application in practice, mainly due to the heavy computational burden where the knowledge of conditional probability distributions are required. Both histogram and kernel based estimation methods have computational challenges for entropy-based sensitivity indices \citep{pianosi2015simple}. 

Motivated by the above issues and inspired by the low-cost sensitivity screening proxy for variance-based measures, in the next section, we will propose a computationally efficient upper bound for the entropy-based total sensitivity measure. We then extend the upper bound to the DGSM indices via exponentiation, and show that in the special case with Gaussian inputs and linear functions, the proposed new proxy for entropy GSA is equivalent to the proxy for variance-based total effect sensitivity. 
\begin{figure}[!t]
	\centering
	\includegraphics[width=16 cm]{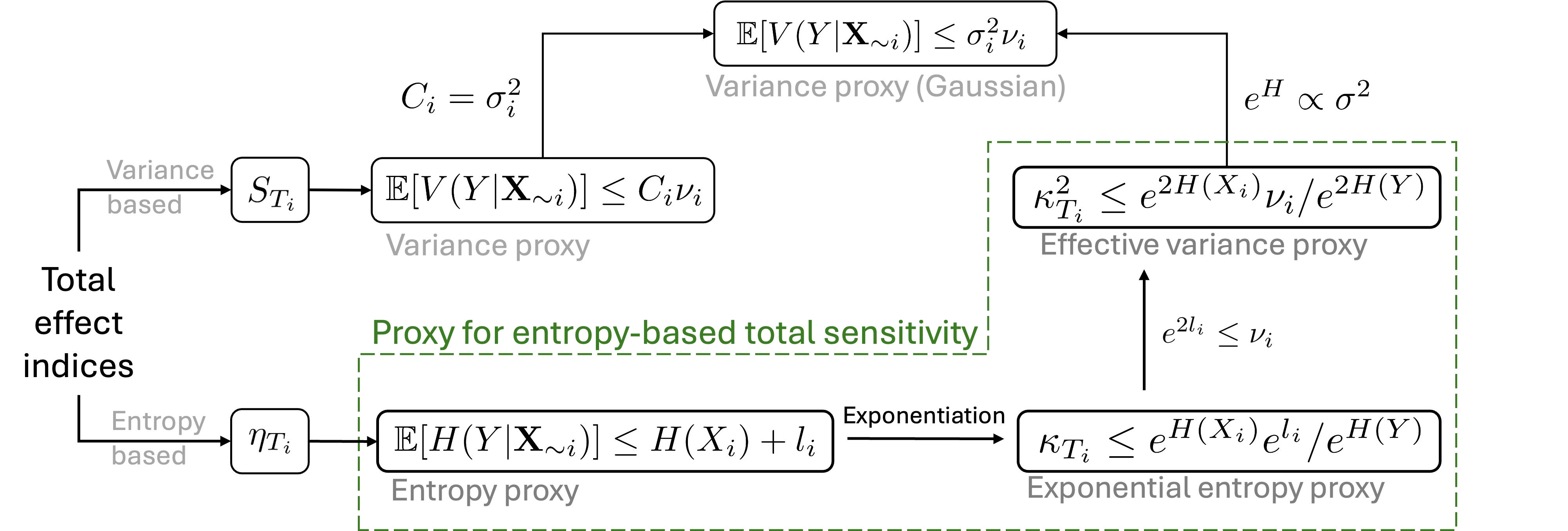}
	\caption{An overview for the relationship between entropy and variance proxies where the entropy proxies developed in this paper are highlighted in the box with dash lines.}
	\label{fig:flowchart}
\end{figure}

An overview of the above mentioned relationship between entropy and variance proxies is summarised in Figure \ref{fig:flowchart}. The sensitivity proxies are represented by the upper bounds of the inequalities in Figure \ref{fig:flowchart}, where $l_i$ and $\nu_i$ are derivative-based quantities. The entropy proxies to be developed in the next sections are highlighted in the box with dash lines. 
\section{Link between the partial derivatives and the total effect entropy measure}
\label{sec:3}
In this section we will consider a differentiable function $y = g(\mathbf{x}) $. Recall that our interest here is for applications of global sensitivity analysis (GSA) towards mathematical models, where a physical phenomenon is typically studied with a complex numerical code. The computation of the partial derivatives can then be obtained via the companion adjoint code, or numerically estimated by a finite difference method. For example, the derivative-based DGSM can be estimated by finite difference, and this can be performed efficiently via Monte Carlo sampling as discussed in \cite{kucherenko2009monte}.
\subsection{An upper bound for the total effect entropy}
\label{sec:3.1}
For a general vector transformation $\mathbf{Y} = \mathbf{g} (\mathbf{X})$, the differential entropy of the output is related to the input via \citep[p.660]{papoulis2002book}:
\begin{equation} \label{eq:Link.1}
   \begin{split}
       H(\mathbf {Y}) &  \leq  H(\mathbf{X}) + \int f_{X}(\mathbf{x}) \ln {\left| \det{\mathbb{J}}\right|} d\mathbf{x}
  \end{split}
\end{equation}
where $\mathbb{J}$ is the Jacobian matrix with $\mathbb{J}_{ij} = \partial g_i / \partial x_j$ and $f_{X}(\mathbf{x})$ is the probability density function (PDF) of $\mathbf{X}$. $\mathbf{g(X)}$ is assumed to be differentiable and the partial derivatives are assumed to be square integrable. The above inequality becomes an equality if the transform is a bijection, i.e. an invertible transformation. Note that there is no independence assumption for the inputs of the inequality above. 

As shown in \cite{papoulis2002book}, Eq \ref{eq:Link.1} can be proved substituting the transformed PDF $f_Y(\mathbf{y}) = f_X(\mathbf{x})/\det{\mathbb{J}}$ into the expression of $H(\mathbf{Y}) = - \int f(\mathbf{y}) \ln f(\mathbf{y}) d\mathbf{y}$ and note there will be a reduction of entropy if the transformation is not one-to-one. Following this line of thought, Eq \ref{eq:Link.1} can also be seen as one version of the data processing inequality, where the transformation does not increase information \citep{ geiger2011information}. 

Given the data processing inequality in Eq \ref{eq:Link.1}, we have the following theorem to upper bound the total effect entropy:
\begin{theorem}
For a differentiable deterministic function $y = g(\mathbf{x}): \mathbb{R}^d \rightarrow \mathbb{R}$ with continuous random inputs, there exists an inequality for the total effect entropy:
\begin{equation} \label{eq:proposition}   
       H_{T_i} \leq H(X_i) + l_i
\end{equation}
where $H_{T_i}$ is the total effect entropy which is an expected conditional entropy $H_{T_i} = \mathbb{E}[H(Y|\mathbf{X}_{ \sim i})]$, where $\sim i$ indicates the index ranges from $1$ to $d$ excluding $i$. $H(X_i)$ is the differential entropy of the input variable $X_i$ and $l_i$ is the expected log-derivatives $ l_i =  \mathbb{E} \left[ \ln {\left| {\partial g (\mathbf{x})}/{\partial x_i} \right|} \right]$. 
\end{theorem}
\begin{proof}
Set $y_1 = g_1(\mathbf{x}) = g(\mathbf{x})$ where $\mathbf {x} = (x_1,x_2, \dots, x_d)$, and introduce dummy variables $y_i = g_i(\mathbf{x})=x_i$ with $i = 2, \dots, d$. In this setting, the Jacobian matrix from the 2nd row onwards, i.e. $i \geq 2$,  $\partial g_i / \partial x_j = 1$ when $i=j$ and   $\partial g_i / \partial x_j = 0$ when $i \neq j$. Therefore, the Jacobian matrix in this case is a triangular matrix. As a result, the Jacobian determinant is the product of the diagonal entries: 
\begin{equation*} \label{eq:Link.2}
\det{\mathbb{J}} = \left| \frac {\partial g_1}{\partial x_1} \times
\underbrace{1 \times \dots \times 1 }_{2 \, \text{to}  \, d}
\right| 
= \left| \frac {\partial g_1}{\partial x_1} \right|
\end{equation*}

The information processing inequality from Eq \ref{eq:Link.1} can thus be expressed as: 
\begin{equation*} \label{eq:Link.3}
   \begin{split}
       H(\mathbf {Y}) \leq H(\mathbf{X}) + \int f_{X}(\mathbf{x}) \ln {\left| \frac {\partial g_1(\mathbf{x})}{\partial x_1} \right|} d\mathbf{x}
  \end{split}
\end{equation*}
where $\mathbf{Y} = \{ Y, X_2, X_3, \dots, X_d \}$. 

On the left hand side of the above inequality, the joint entropy of $\mathbf{Y}$ can be expressed using the conditional entropies as $H(\mathbf{Y}) = \mathbb{E}[H(Y|\mathbf{X}_{\sim 1}) ]+ H(\mathbf{X}_{\sim 1})$ using the chain rule for differential entropies \citep{cover1999elements}. On the right hand side, we have $H(\mathbf{X}) \leq H(X_1) + H(\mathbf{X}_{ \sim 1})$ using the subadditivity property of the joint entropy of the input variables. The joint entropy $H(\mathbf{X}) $ becomes additive if the input variables are independent. 

Putting these together, the above inequality using the variable $x_1$ then becomes:
\begin{equation*} 
       \mathbb{E}[H(Y|\mathbf{X}_{\sim 1})] \leq H(X_1) + l_1
\end{equation*}
The reasoning above uses the first variable $x_1$ as an example. However, the results hold for any variables via simple row/column exchanges, which only affects the sign of the determinant but not its modulus. 
\end{proof}

The total effect entropy $H_{T_i} = \mathbb{E}[H(Y|\mathbf{X}_{\sim i})]$ is a global measure of uncertainty as the expectation is with respect to all possible values of $\mathbf{X}_{\sim i}$. It measures the remaining entropy of $Y$ if the true values of $\mathbf{X}_{\sim i}$ can be determined, in analogy to the total effect variance $V_{T_i}$.

The total effect entropy inequality from Eq \ref{eq:proposition} thus demonstrates that, for a differentiable function $y = g(\mathbf{x})$, the entropy-based total sensitivity is bounded by the expectation of log partial derivatives of the function, with the addition of the entropy of the input variable of interest. And this inequality becomes an equality if the input variables are independent and the transformation $g(\cdot)$ has a unique inverse. 

The inequality in Eq \ref{eq:proposition} is one of the main contributions of this paper. It establishes an upper bound for the total effect entropy $H_{T_i}$ for ETSI given in Eq \ref{eq:Entropy.total} using computationally efficient partial-derivative based functionals. As smaller $l_i + H(X_i)$ tends to indicate smaller total effect entropy, it can thus be used to screen un-influential variables and work as a low cost proxy for entropy-based indices. 
 subsection{Exponential entropy based total sensitivity measure}
\label{sec:3.2}
The use of  $l_i + H(X_i)$ as a screening proxy for the total effect entropy is similar to the DGSM-based upper bound for the variance-based $S_{T_i}$ described in Eq \ref{eq:STi_DGSM}. However, there are two issues with the differential entropy based sensitivity measure: 1) as pointed out in \cite{kala2021global}, entropy for continuous random variables (aka differential entropy) can become negative and this is undesirable for sensitivity analysis. More importantly,  the inequality in Eq \ref{eq:proposition} is not valid when normalised by a negative $H(Y)$; 2) the interpretation of conditional entropy is not as intuitive as variance based sensitivity indices. This is partly due to the fact that variance-based methods is firmly anchored in variance decomposition, but also because entropy measures the average information or non-uniformity of a distribution as compared to variance which measures the spread of data around the mean. Although non-uniformity can be seen as a suitable measure for epistemic uncertainties \citep{krzykacz2001epistemic}, its interpretation for GSA in a general setting is less intuitive. 

To overcome these two issues, we propose to use exponential entropy as a entropy-based measure for global sensitivity analysis. Although not directly investigated, studies in \cite{auder2008entropy} have noted that an exponentiation of the standard entropy-based sensitivity measures may improve its discrimination power. We take an exponentiation of the total effect entropy inequality in Eq \ref{eq:proposition}:
\begin{equation} \label{eq:ExponentialBound}   
       e^{H_{T_i}} \leq  e^{H(X_i)} e^{l_i}
\end{equation}
where we recall that $l_i =  \mathbb{E} \left[ \ln {\left| \partial g(\mathbf{x})/\partial x_i \right|} \right]$. 
Divide both sides of Eq \ref{eq:ExponentialBound} by $e^{H(Y)}$:
\begin{equation} \label{eq:kappa}   
       \kappa_{T_i} = \frac{e^{H_{T_i}}} {e^{H(Y)}} \leq  \frac{e^{H(X_i)}}{e^{H(Y)}} e^{l_i}
\end{equation}
where $\kappa_{T_i}$ can be considered as the exponential entropy based total sensitivity indices (eETSI), and the upper bound can then be used as a proxy for $\kappa_{T_i}$ to detect less influential input variables. As the total effect entropy $H_{T_i}=\mathbb{E}[{H(Y|\mathbf{X}_{ \sim i})}] \leq H(Y)$, we then have $0 < \kappa_{T_i} \leq 1$ which is desirable as sensitivity indices. 

In addition, the exponential entropy based eETSI $\kappa_{T_i}$ and the un-normalised $e^{H_{T_i}}$ have a more intuitive interpretation as GSA indices as compared to the standard differential entropy, because exponential entropy can be seen as a measure for the effective spread or extent of a distribution \citep{campbell1966exponential}. We will explain this intuition using several simple examples in the next section. 

\subsection{Intuitive interpretation of sensitivity based on exponential entropy}
\label{sec:3.3}
To explain the intuitive interpretation of exponential entropy based sensitivity indices, we first recall that the entropy of a random variable with a uniform distribution is $\ln(b-a)$, where $a, b$ are the bounds of the distribution. Taking the natural exponentiation of the entropy in this case results in $b-a$, which is the range of the uniform distribution. For a Gaussian distribution with variance $\sigma^2$, the exponential entropy is $\sqrt{2\pi e}\sigma^2$ which is proportional to the variance. 
\begin{figure}[!h]
	\centering
	 \subfloat[\centering Example 1] {{\includegraphics[width=5cm]{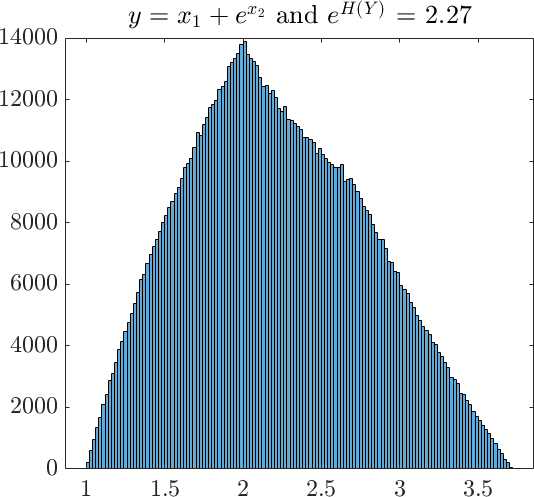} } \label{fig:eHYa}} 
	 \hfill
    \subfloat[\centering Example 2]{{\includegraphics[width=4.5cm]{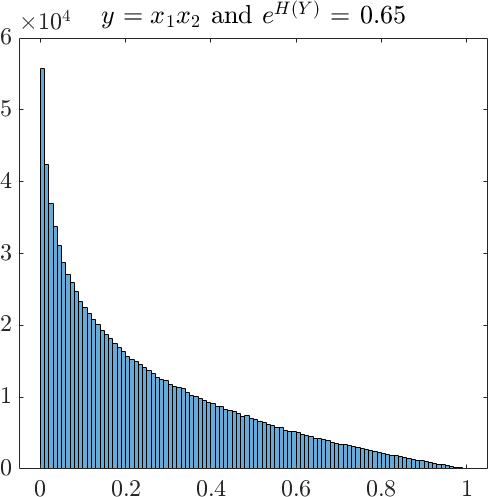} }\label{fig:eHYb}}
    	 \hfill
	  \subfloat[\centering Example 3]{{\includegraphics[width=5cm]{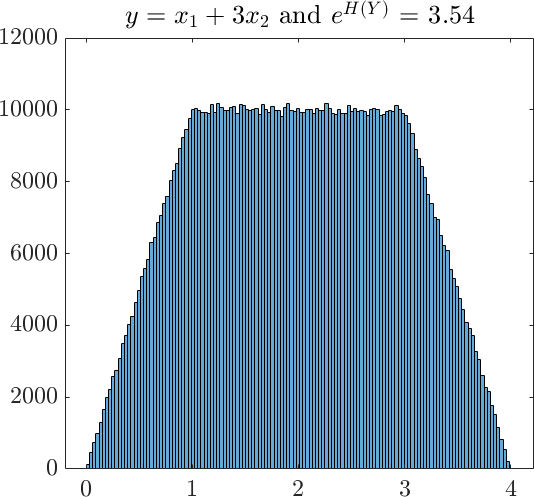} }\label{fig:eHYc}}
	\caption{Histograms of the output for monotonic examples 1 to 3 as described in Section 4 of the \textit{Main Text}. The exponential entropy of the outputs are (a) $e^{H(Y)} =2.26$; (b) $e^{H(Y)} =0.65$; (c) $e^{H(Y)} = 3.54$, which describes the effective extent/support of the underlying distribution. In comparison, the standard deviations of the output are 0.57, 0.22 and 0.91 for the three examples respectively.}
	\label{fig:examples_eHY}
\end{figure}

This becomes more evident if we plot the histogram of function output $Y$ from several simple functions in Figure \ref{fig:examples_eHY}. The histograms of the function outputs are based on $10^6$ Monte Carlo samples and from which we have calculated the exponential entropy $e^{H(Y)}$ for each case. We use examples 1 to 3 in Table \ref{tab:monotonic-ExList} from Section \ref{sec:4.1}, where more details of the entropy calculation can be found. 

When compared to the range of the distribution as represented by the histograms in Figure \ref{fig:examples_eHY}, the values of the exponential entropy provide intuitive indication of the spread of the distribution, in analogy to standard deviation or variance. For example, the average width of the top and bottom of the trapezium in Example 3 is approximately $(3+4)/2 = 3.5$, which is about the same as the exponential entropy of $3.543$. 

Therefore, the exponential entropy $e^H$ can be regarded as a measure of the extent, or effective support, of a distribution and this has been discussed in details in \cite{campbell1966exponential}. As the total effect entropy $H_{T_i}$ measures the remaining entropy in average if the true values of $\mathbf{X}_{\sim i}$ can be determined, $e^{H_{T_i}}$ can thus be regarded as the effective remaining range of the output distribution conditioning on that $\mathbf{X}_{\sim i}$ are known. The normalised indices eETSI $\kappa_{T_i}$ then measure the ratio of the effective range before and after $\mathbf{X}_{\sim i}$ are fixed, and larger $\kappa_{T_i}$ thus indicate a higher influence of $X_i$. 
\subsection{Link between exponential entropy and variance}
\label{sec:3.4}
In addition to its non-negativity and a more intuitive interpretation for GSA, exponential entropy is also closely linked to variance-based GSA indices and their corresponding bounds. To demonstrate this, we first note that the three different derivative-based sensitivity indices are closely related as:
\begin{equation} \label{eq:Deri.link}
   e^{l_i} \leq \mu_i \leq \sqrt{\nu_i}
 \end{equation}
where we recall $l_i =  \mathbb{E} \left[ \ln {\left| {\partial g (\mathbf{x})}/{\partial x_i} \right|} \right]$, $\mu_i =  \mathbb{E} \left[ {\left| {\partial g (\mathbf{x})}/{\partial x_i} \right|} \right]$ and $ \nu_i =  \mathbb{E} \left[ {\left| {\partial g (\mathbf{x})}/{\partial x_i} \right|^2} \right]$. 

It is evident that $\mu_i \leq \sqrt{\nu_i}$ based on Cauchy-Schwarz inequality. In addition, we have $e^{l_i} \leq \mu_i$ using Jensen's inequality as the exponential function is convex. So the inequality for the exponential entropy based eETSI from Eq \ref{eq:kappa} can be further associated with DGSM as: 
\begin{equation} \label{eq:kappa_DGSM}   
        \kappa^2_{T_i} = \frac{e^{2H_{T_i}}} {e^{2H(Y)}} \leq  \frac {e^{2H(X_i)}}{e^{2H(Y)}} \nu_i
\end{equation}
where we recall that $\nu_i$ are the derivative-based DGSM indices. 

Eq \ref{eq:kappa_DGSM} already looks remarkably similar to the variance-DGSM inequality given in Eq \ref{eq:STi_DGSM}. In fact, the squared exponential entropy $e^{2H(X)}$ of a random variable $X$, also called entropy power, is known from information theory to be bounded by the variance of $X$ with $e^{2H(X)} \leq 2\pi e {V(X)}$ \citep{cover1999elements}, where the equality is obtained if $X$ is a Gaussian random variable. Therefore, for independent Gaussian inputs with variance $\sigma_i^2$, we have $e^{2H(X_i)} = 2\pi e \sigma_i^2 $. The squared exponential entropy $e^{2H(X)}$ thus measures the `effective variance' as it is simply the variance of a Gaussian random variable with the same entropy \citep{costa1984similarity}.

In the special case where the function $y = g(\mathbf{x})$ is a linear function, the conditional distribution of $Y|\mathbf{X}_{\sim i}$ is also Gaussian, and it indicates that the conditional entropy power would also be proportional to the conditional variances of the output, i.e. $e^{2H_{T_i}}  = 2\pi e V_{T_i}$ (cf Example 5 in Section \ref{sec:4.1}). Therefore, it is clear that with independent Gaussian inputs and a linear function, we have from Eq \ref{eq:kappa_DGSM} that:
\begin{equation} \label{eq:EntropyVarianceEqui}   
      2\pi e V_{T_i} \leq  2\pi e \sigma_i^2 \nu_i  \quad \rightarrow \quad V_{T_i} \leq  \sigma_i^2 \nu_i 
\end{equation}
by substituting entropy power $e^{2H_{T_i}}$ and $e^{2H(X_i)}$ by their variance counter parts. Eq \ref{eq:EntropyVarianceEqui} thus indicates that in this special case the $\nu_i$-based entropy upper bound from Eq \ref{eq:kappa_DGSM} is equivalent to the $\nu_i$-based variance upper bound relationship given in Eq \ref{eq:STi_DGSM}. 
\section{Numerical illustrations}
\label{sec:4}
In this section, we first examine the special equality case with monotonic functions, and then provide assessment of the total effect entropy inequality with general nonlinear functions. A physical example will be considered in Section \ref{sec:5}. Note that although the inequality in Eq \ref{eq:proposition} makes no assumptions of independence, for simplicity the input variables are assumed to be independent in these numerical examples. The upper bound (UB) can also be adjusted to work with groups of input
variables, and this is illustrated briefly in Appendix \ref{appendix:e}. 

\subsection{Analytical verifications for equality cases with monotonic functions}
\label{sec:4.1}
In this section, we verify analytically that the total effect entropy inequality from Eq \ref{eq:proposition} is tight for a monotonic function. A function $g : \mathbb{R}^d \rightarrow \mathbb{R}$ is monotonically increasing if $x_i \leq x'{_i} $ for all $i$ implies $g(\mathbf{x}) \leq g(\mathbf{x'}) $. Based on the analytical results, we present numerical experiments to investigate the required sample size for the tight bound. 
\begin{figure}[!h]
	\centering
	\includegraphics[width=14 cm]{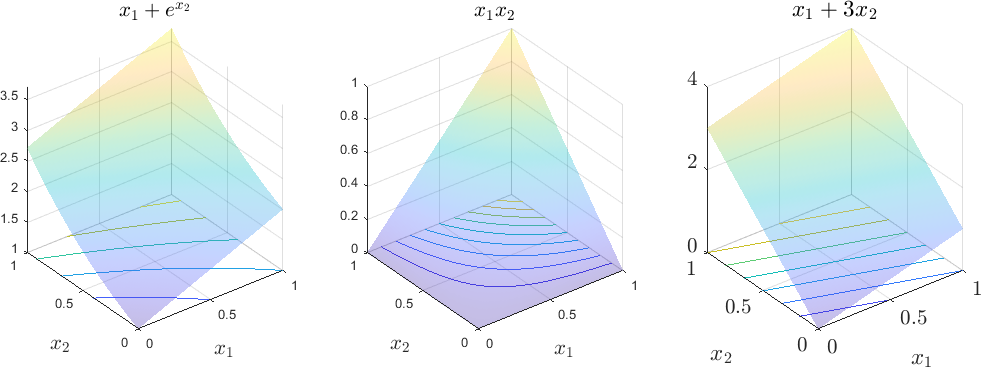}
	\caption{Surface plots for the monotonic functions in examples 1 - 3}
	\label{fig:Monotonic}
\end{figure}

Five monotonic functions are considered. These are listed in Table \ref{tab:monotonic-ExList} and Figure \ref{fig:Monotonic} shows plots of examples 1 - 3 which are non-decreasing in the domain of interest.  All input variables are assumed to have the same uniform distribution for examples 1 - 4, i.e. $x_i \sim \mathbb{U}(0,1)$, while Gaussian distributions are used for example 5. For verification purposes, all the examples in this section are chosen to have tractable expressions for both the integral of derivatives and the conditional entropies. 
\begingroup
\def\arraystretch{0.5}
\setlength{\tabcolsep}{4pt}
\begin{table}[!ht]
\caption{Analytical results for five different monotonic functions, where derivations are given in Appendix \ref{appendix:b}. Note that $ H_{T_i} = \mathbb{E}[H(Y|\mathbf{X}_{ \sim i})] $ is the total effect entropy and $ l_i =  \mathbb{E} \left[ \ln {\left| {\partial g (\mathbf{x})}/{\partial x_i} \right|} \right]$}
\smallskip
    \centering
    \small
    \makebox[\textwidth][c] {
    \begin{tabular}{llcccc}
     \cline{3-6}
  & & \multicolumn{2}{|c|}{$X_1$} &
  \multicolumn{2}{c|}{$X_2$} \\ 
  \hline
\multicolumn{2}{|l|}{Examples} &
  \multicolumn{1}{l|}{$H_{T_1}$} &
  \multicolumn{1}{l|}{$H(X_1) + l_1$} &
  \multicolumn{1}{l|}{$H_{T_2}$} &
  \multicolumn{1}{l|}{$H(X_2) + l_2$} \\ \hline
\multicolumn{1}{|l|}{Ex-1} &
  \multicolumn{1}{l|}{$y=x_1 + e^{x_2}$} &
  \multicolumn{1}{c|}{0} &
  \multicolumn{1}{c|}{0} &
  \multicolumn{1}{c|}{1/2} &
  \multicolumn{1}{c|}{1/2} \\ \hline
\multicolumn{1}{|l|}{Ex-2} &
  \multicolumn{1}{l|}{$y=x_1 x_2$} &
  \multicolumn{1}{c|}{-1} &
  \multicolumn{1}{c|}{-1} &
  \multicolumn{1}{c|}{-1} &
  \multicolumn{1}{c|}{-1} \\ \hline
\multicolumn{1}{|l|}{Ex-3} &
  \multicolumn{1}{l|}{$y=x_1 + 3x_2$} &
  \multicolumn{1}{c|}{0} &
  \multicolumn{1}{c|}{0} &
  \multicolumn{1}{c|}{$\ln{3}$} &
  \multicolumn{1}{c|}{$\ln{3}$} \\ \hline
\multicolumn{1}{|l|}{Ex-4} &
  \multicolumn{1}{l|}{$y= x_1x_2^r$} &
  \multicolumn{1}{c|}{$-r$} &
  \multicolumn{1}{c|}{$-r$} &
  \multicolumn{1}{c|}{$\ln{r} - r $} &
  \multicolumn{1}{c|}{$\ln{r} - r$} \\ \hline
\multicolumn{1}{|l|}{Ex-5} &
  \multicolumn{1}{l|}{$y= \sum_{i=1}^d a_i x_i$} &
  \multicolumn{1}{c|}{$\ln{|a_i|}$} &
  \multicolumn{1}{c|}{$\ln{|a_i|}$} &
  \multicolumn{2}{l|}{for all $i$} \\ \hline
      \end{tabular}
    }
    	\label{tab:monotonic-ExList}
\end{table}
\endgroup

From table \ref{tab:monotonic-ExList} we can see that $H_{T_i} = H(X_i) + l_i$ for the monotonic examples considered. These analytical results not only verify that the inequality from Eq \ref{eq:proposition} is tight for monotonic functions with independent inputs, but also provide benchmark for convergence test of numerical estimations. 

For examples 1 - 3, the total effect entropies $H_{T_i}$ are also numerically estimated using the method given in Appendix \ref{appendix:a}. We estimate $H_{T_i}$ with Monte Carlo sampling, with number of samples ranging from $10^3$ to $10^8$ as shown in Figure \ref{fig:HTi_mono}.  Also shown are the standard deviations (std) from 10 repeated estimations and the analytical values from Table \ref{tab:monotonic-ExList}. 

It can be seen from Figure \ref{fig:HTi_mono} that the estimation of $H_{T_i}$ converges to the exact values with increasing number of samples, and the relative error with $10^8$ samples is less than 1\% for all functions. However, large number of samples is required. 

In comparison, a smaller number of samples is sufficient for the estimation of derivative-based sensitivity measures. Numerical error of $l_i$ with 100 samples is less than 1\% for the low dimensional cases considered. Note that the number of samples required is case dependent and a typical cost of derivative-based indices is in the order $10^3$ or $10^4$ \citep{kucherenko2009monte}. For the estimation of $l_i$, the finite difference method for approximating the partial derivative is used with a fixed increment step of $10^{-5}$ following \cite{kucherenko2009monte} for DGSM estimation. Note that the numerical results for $l_i$ are not shown in Figure \ref{fig:HTi_mono} as it is indistinguishable from the analytical values in comparison to $H_{T_i}$. 
\begin{figure}[!t]
    \captionsetup[subfloat]{labelformat=empty}
	\centering
	 \subfloat[ ] {{\includegraphics[width=5.5cm]{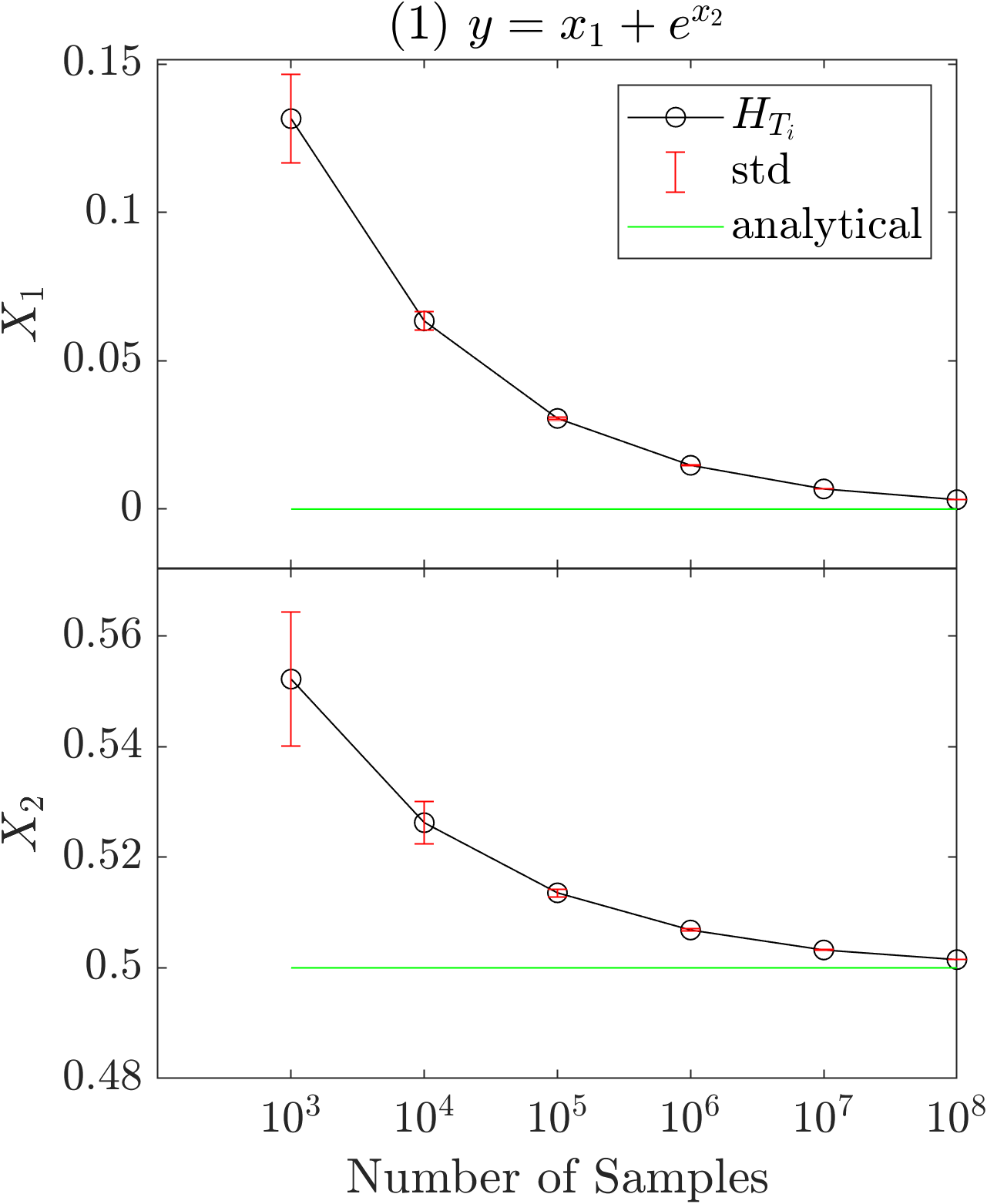} } \label{fig:Ex-1}} 
		 \subfloat[ ] {{\includegraphics[width=5cm]{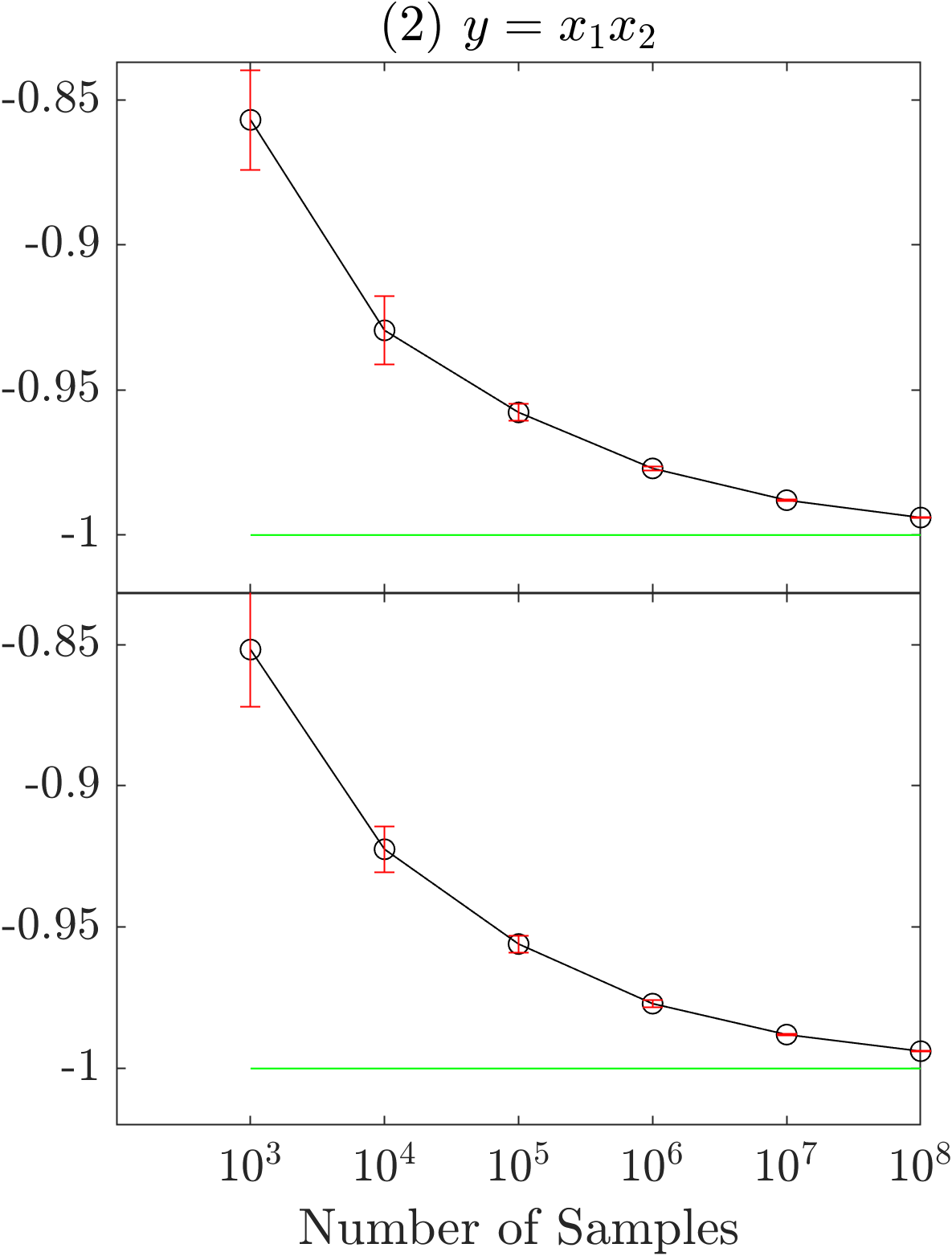} } \label{fig:Ex-2}} 
		 	 \subfloat[] {{\includegraphics[width=5cm]{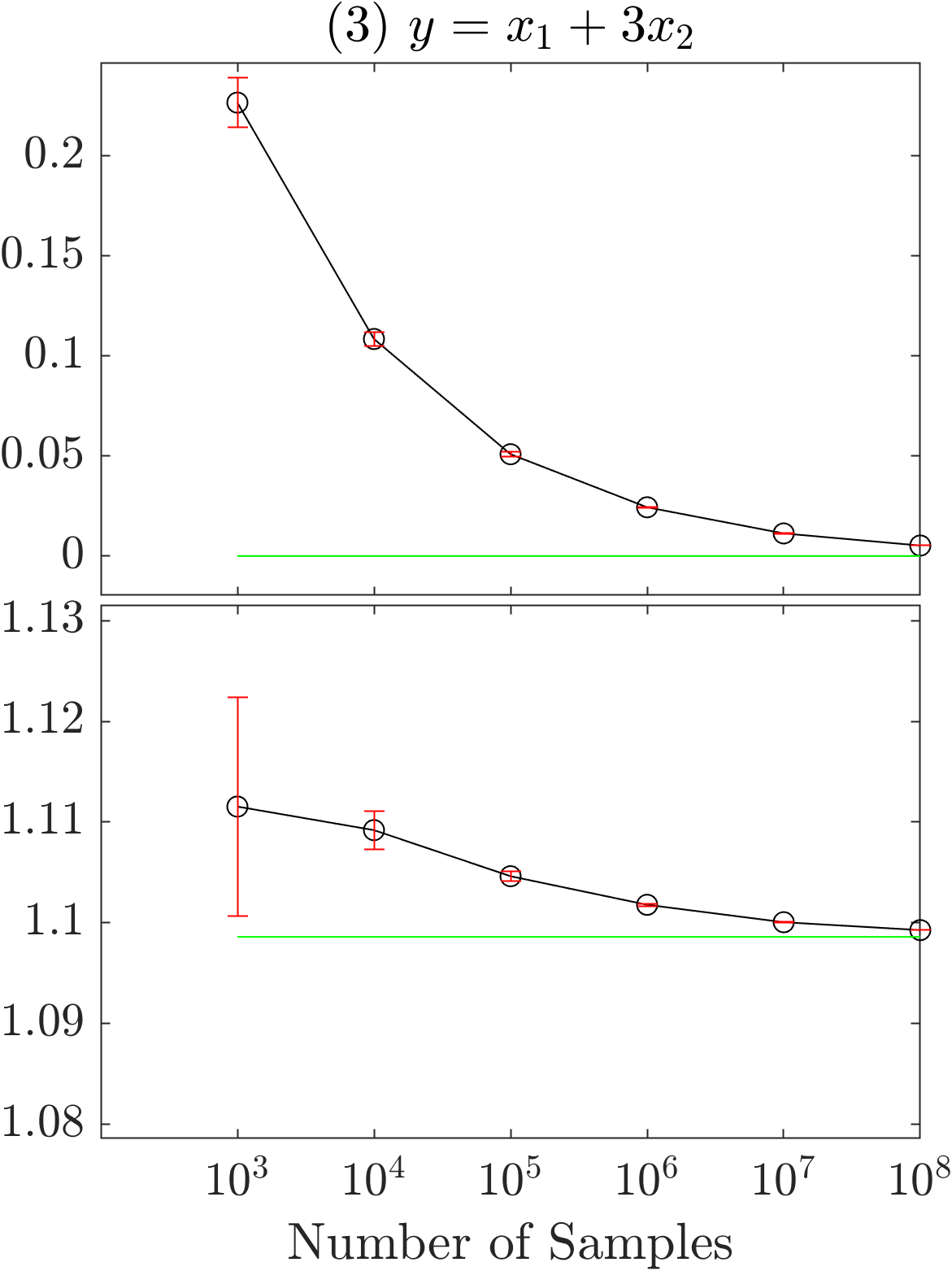} } \label{fig:Ex-3}} 
		\vspace*{-8mm}
	 	\caption{ Convergence of the numerical estimated total effect entropy for monotonic examples 1 - 3.}
	\label{fig:HTi_mono}
\end{figure}

\subsection{Illustrations with Ishigami function and G-function}
\label{sec:4.2}
In this section, we use Ishigami function (Figure \ref{fig:Ishigami}) and G-function (Figure \ref{fig:Gfunction}) for illustrations with general nonlinear functions. Both functions are commonly used test functions for global sensitivity analysis, due to the presence of strong interactions. 
\begin{figure}[!t]
	\centering
	 \subfloat[\centering Ishigami function] {{\includegraphics[width=10cm]{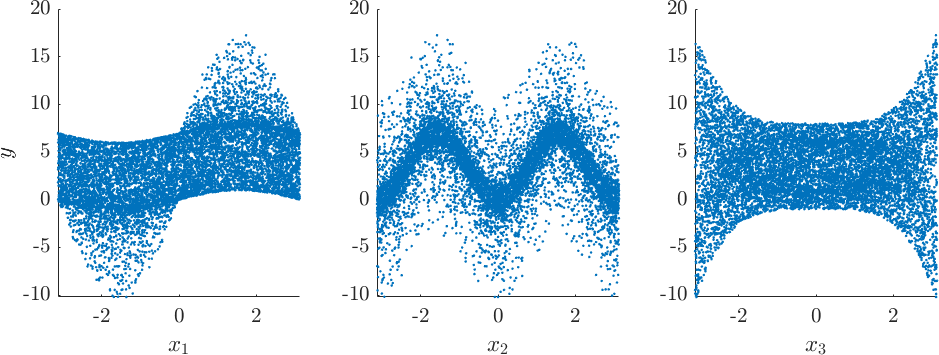} } \label{fig:Ishigami}} 
	 \quad
	 \subfloat[\centering G-function] {{\includegraphics[width=4cm]{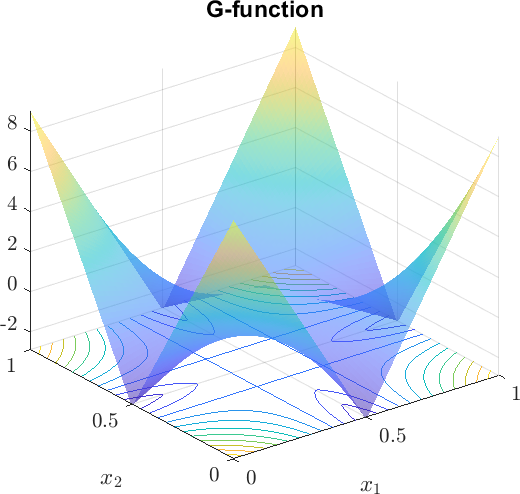} } \label{fig:Gfunction}} 
	 	\caption{ Example plot for the general functions considered. a) Scatter plot for the Ishigami function, $y = \sin(x_1) + 7 \sin^{2} (x_2) + 0.1 x^{4}_3\sin(x_1)$, $x_i \sim \mathbb{U} (-\pi, \pi)$ for $i = 1, 2, 3$. ; b) Surface plots for the G-function (2 variable plot), $y = \prod_{i=1}^{3} (|4x_i - 2| + a_i )/(1+a_i)$, $x_i \sim \mathbb{U} (0, 1)$ for $i = 1, 2, 3$. In this case, $a_i = (i-2)/2$, for  $i = 1, 2, 3$. A lower value of $a_i$ indicates a higher importance of the input variable $x_i$, i.e., $x_1$ is the most important, while $x_3$ is the least important in this case}
	\label{fig:generalfunction}
\end{figure}

\subsubsection{Assessment of total effect entropy inequality}
\label{sec:4.2.1}
These two functions, each with three input variables, are used in this section to demonstrate the inequality relationship derived in Eq \ref{eq:proposition}. The conditional entropies are estimated numerically using Monte Carlo sampling as described in Appendix \ref{appendix:a}. Different numbers of samples are used, ranging from $1e6$ to $1e8$. For each estimation, the computation is repeated 20 times and both the mean value and the standard deviation (std) are reported in Table \ref{tab:ishigami} and \ref{tab:G}. For the estimation of derivative-based $l_i$, we use the finite-difference based approach with 1000 samples. As the analytical expressions of the partial derivatives are readily available in this case, the derivative-based results are further verified with direct integration using Matlab's inbuilt numerical integrator "integral" with default tolerance setting. 

The sensitivity results for the Ishigami function are listed in Table \ref{tab:ishigami}, where it is clear that inequality from Eq \ref{eq:proposition} is satisfied. It is clear the standard deviation is small. However, the convergence of conditional entropy estimation is slow as large number of samples is needed, as we saw for monotonic functions in Figure \ref{fig:HTi_mono}. 
\begingroup
\def\arraystretch{0.3}
\begin{table}[!h]
	\caption{Total effect entropy results for the Ishigami function, which are obtained for different number of samples. This is repeated for 20 times and the mean and standard deviation (std) are given. The results from $10^8$ samples are compared to $H(X_i) + l_i$, for which the inequality in Eq \ref{eq:proposition} is clearly satisfied.}
	\smallskip
    \centering
    \makebox[\textwidth][c] { 
    \begin{tabular}{ccccccc}
    \toprule
            ~& \multicolumn{6}{c} {Ishigami function } \\
        ~& \multicolumn{6}{c} {$y = \sin(x_1) + 7\sin^{2} (x_2) + 0.1 x^{4}_3\sin(x_1)$ } \\
        \cmidrule(lr){2-7} 
           \multirow{2}{4cm}{Number of Samples}   & \multicolumn{2} {c} {$H_{T_1}$} & \multicolumn{2} {c} { $H_{T_2}$} & \multicolumn{2} {c} {$H_{T_3}$} \\ 
           & mean & std & mean & std & mean & std \\
        	\cmidrule(lr){2-3} \cmidrule(lr){4-5} \cmidrule(lr){6-7}
1.00E+06 & 1.3902 & 0.0007 & 1.7614 & 0.0006 & 0.9701 & 0.0013 \\
1.00E+07 & 1.2978 & 0.0003 & 1.7023 & 0.0001 & 0.7693 & 0.0004 \\
1.00E+08 & 1.2335 & 0.0001 & 1.6609 & 0.0001 & 0.6066 & 0.0002 \\
      \\
          &  $X_1$ & $X_2$ & $X_3$ &  & \\
           \cmidrule(lr){2-4} 
        $H_{T_i} = \mathbb{E} [H(Y|\mathbf{X}_{\sim i})]$  &  1.2335 & 1.6609 & 0.6066     & &   \\
           $H(X_i) + l_i$ & 1.9024 & 3.0906 & 0.6626 & &\\
        \bottomrule
    \end{tabular}
    }
    	\label{tab:ishigami}
\end{table}
\endgroup

The sensitivity results for the G-function, in the same format as Table \ref{tab:ishigami}, are reported in Table \ref{tab:G}. It is clear that the inequality relationship in Eq \ref{eq:proposition} is also satisfied. The results in Table \ref{tab:G} also highlights the issue that the differential entropy based results can be negative. The relative amplitudes of $H_{T_i}$ can still indicate the relative importance of the input variables for the output entropy, but the negative amplitudes are undesirable for sensitivity analysis. 
\begingroup
\def\arraystretch{0.3}
\begin{table}[!h]
	\caption{Sensitivity results for the G-function, where the inequality in Eq \ref{eq:proposition} is clearly satisfied. Same key as Table \ref{tab:ishigami}}
	\smallskip
    \centering
    \makebox[\textwidth][c] { 
    \begin{tabular}{ccccccc}
    \toprule
        ~& \multicolumn{6}{c} { G-function  \quad $\displaystyle y = \prod_{i=1}^{3} \frac{|4x_i - 2| + a_i }{1+a_i}$ with $a_i = (i-2)/2$} \\
        \cmidrule(lr){2-7} 
            \multirow{2}{4cm}{Number of Samples}   & \multicolumn{2} {c} {$H_{T_1}$} & \multicolumn{2} {c} { $H_{T_2}$} & \multicolumn{2} {c} {$H_{T_3}$} \\ 
           & mean & std & mean & std & mean & std \\
        	\cmidrule(lr){2-3} \cmidrule(lr){4-5} \cmidrule(lr){6-7}
    1.00E+06 &  0.3477 & 0.0009 & -0.1376 & 0.0013 & -0.3988 & 0.0015 \\
    1.00E+07 & 0.3398 & 0.0006 & -0.1737 & 0.0005 & -0.4482 & 0.0006 \\
    1.00E+08 & 0.3378 & 0.0003 & -0.1917 & 0.0002 & -0.4738 & 0.0002 \\
       \\
          &  $X_1$ & $X_2$ & $X_3$ &  & \\
           \cmidrule(lr){2-4} 
        $H_{T_i} = \mathbb{E} [H(Y|\mathbf{X}_{\sim i})]$   &  0.3378 & -0.1917 & -0.4738     & &   \\
        $H(X_i) + l_i$ & 1.3863 & 0.9808 & 0.6931 & &\\
        \bottomrule
    \end{tabular}
    }
    	\label{tab:G}
\end{table}
\endgroup
\subsubsection{Ranking with eETSI}
\label{sec:4.2.2}
In this section, we take the exponentiation of the total effect entropy, and discuss the exponential entropy based total sensitivity index (eETSI) $\kappa_{T_i}$ and its $l_i$-based upper bound (UB). $\kappa_{T_i}$ are obtained using the mean values of the total effect entropy $H_{T_i}$ from Table \ref{tab:ishigami} and \ref{tab:G} with $10^8$ samples, and the corresponding $e^{H(Y)}$ for the output. 

The results of the sensitivity indices are shown in Figure \ref{fig:Rank_G_Ishigami} for both Ishigami function and G-function. It can be seen that $l_i$-based upper bounds provide the same variable ranking as $\kappa_{T_i}$, although the bound can be loose for these non-linear functions. 
\begin{figure}[!t]
	\centering
	 \subfloat[\centering Ishigami function] {{\includegraphics[width=7cm]{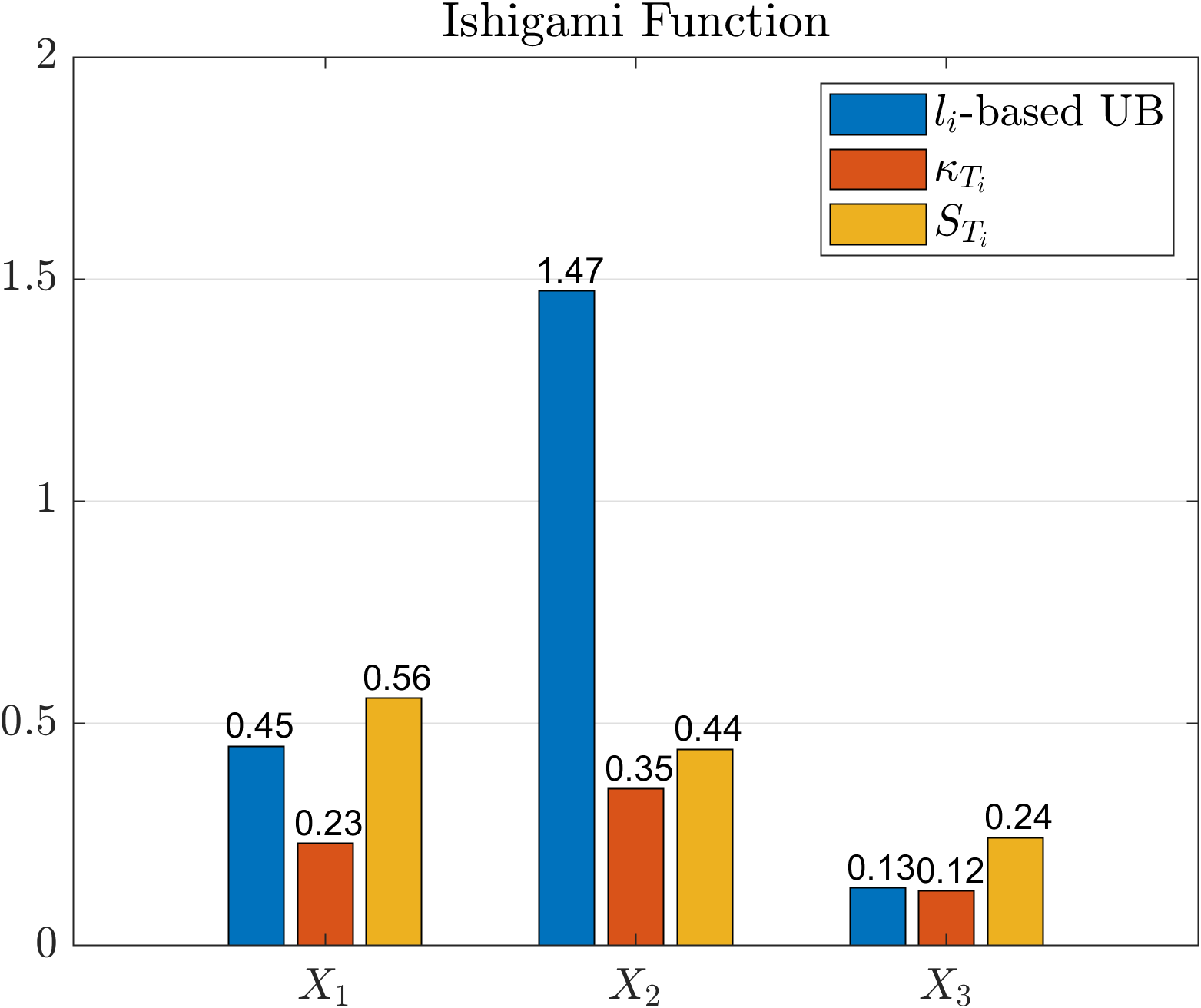} } \label{fig:Ishigami rank}} 
	 \quad
	 \subfloat[\centering G-function] {{\includegraphics[width=7cm]{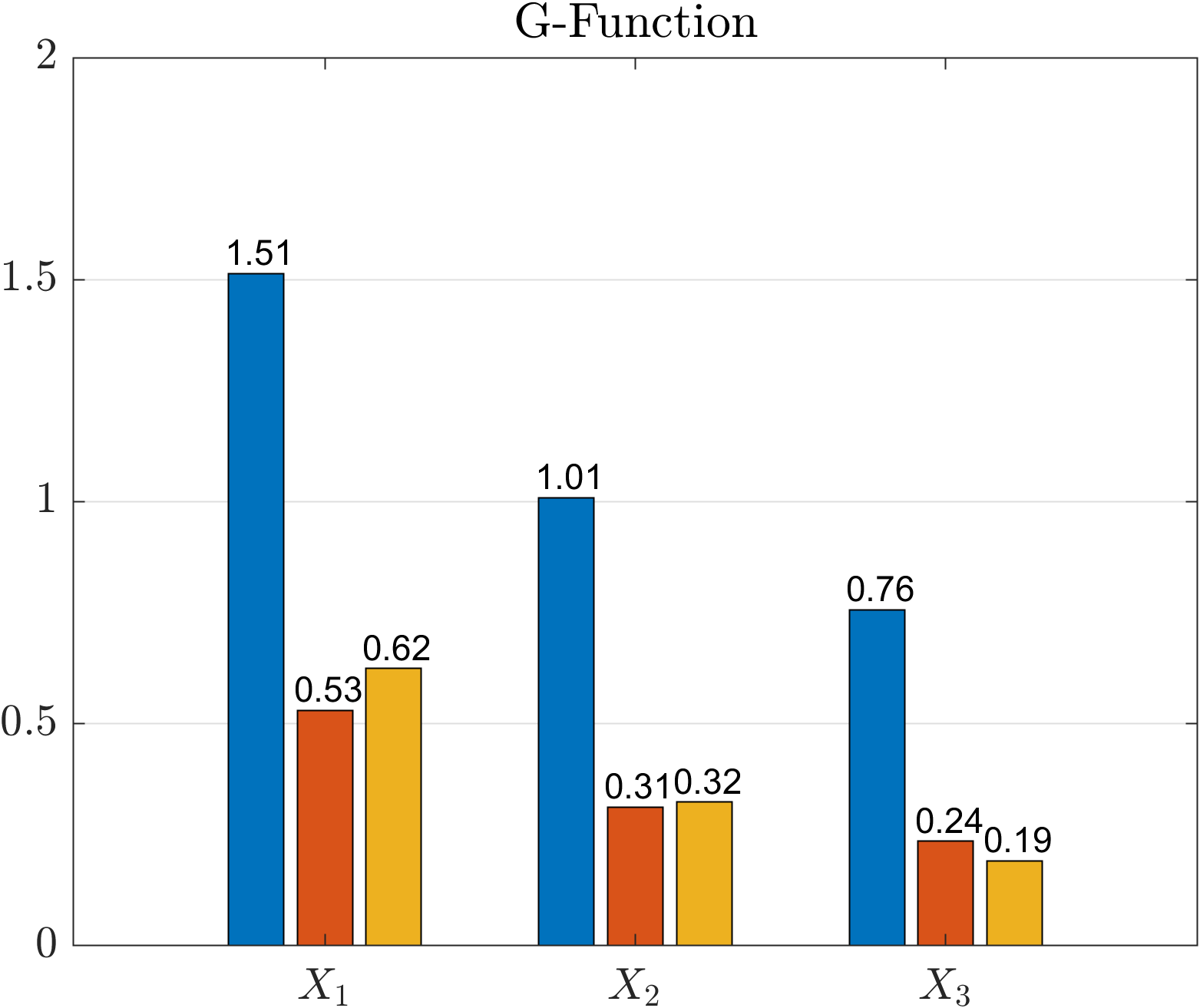} } \label{fig:G rank}} 
	 	\caption{Sensitivity indices for Ishigami function and G-function}
	\label{fig:Rank_G_Ishigami}
\end{figure}

In addition, we have also calculated the variance based $S_{T_i}$, using $10^5$ samples with 20 repetitions, where the mean values are shown in Figure \ref{fig:Rank_G_Ishigami}. It can be seen that the results from $\kappa_i$ are generally consistent with $S_{T_i}$, especially for G-function where the sensitivity ranking are similar both qualitatively and quantitatively. 

For the Ishigami function, both indices have successfully identified the contribution of $X_3$ which is the lowest. However, the relative importance of $X_1$ and $X_2$ are opposite from $\kappa_i$ and $S_{T_i}$. We explain in Appendix \ref{appendix:d} that the interaction between $x_1$ and $x_3$ is more influential for variance due to the squaring effect, as compared to the entropy operation which takes logarithm of the interaction. This difference increases towards the boundary as the interaction between $x_1$ and $x_3$ gets stronger towards $-\pi$ and $\pi$. And this helps to explain why $x_1$ is the most influential variable for the variance-based $S_{T_i}$. This example highlights that, despite many similarities, entropy and variance are fundamentally different, for example, the variable interactions are processed differently between them. Note that the difference between variance-based and entropy-based ranking for Ishigami function was also noted in \cite{auder2008entropy}. 
\subsection{Numerical assessment using a randomised meta-function}
\label{sec:4.3}
We have seen from the numerical assessments that although $l_i$-based upper bounds are tight for monotonic functions, they can be loose for non-linear functions. Nevertheless, the $l_i$-based upper bounds provide same variable rankings as the eETSI $\kappa_{T_i}$, for both Ishigami function and G-function. To further investigate the screening capability of the proposed derivative-based upper bound for entropy sensitivity indices, we adopt a randomised metafunction, i.e. a function generating function, as our test function in this section. 

We follow \cite{becker2020metafunctions} and generate the metafunction as a combination of random selected basis functions that commonly occur in physical models. Nine different basis functions are used, ranging from continuous linear and exponential functions to discontinuous step functions, as listed in Table \ref{tab:MetaF_basis}. The basis functions used are the same as \cite{becker2020metafunctions}. These functions are drawn randomly 1000 times to form the meta function. The histograms of the selected basis functions and the corresponding coefficients are shown in Figure \ref{fig:MetaDist_f} and Figure \ref{fig:MetaDist_coef} respectively. 
\begingroup
\def\arraystretch{0.5}
\setlength{\tabcolsep}{6pt}
\begin{table}[ht]
	\caption{List of basis functions for the meta function}
    \smallskip
    \centering
\makebox[\textwidth][c] {
	\begin{tabular}{ccl} 	
		\toprule
1 & Linear        &     $f^1(x) = x $     \\
2 & Quadratic     &     $f^2(x) = x^2 $      \\
3 & Cubic         &       $f^3(x) = x^3 $    \\
4 & Exponential   &     $f^4(x) = (e^x - 1) / (e - 1)$      \\
5 & Periodic      &      $f^5(x) = 0.5 \sin(2 \pi x) + 0.5$     \\
6 & Discontinuous &    $f^6(x) = 1$ if $x \geq 0.5$ and $0$ otherwise        \\
7 & Dummy         &     $f^7(x) = 0 $      \\
8 & Non-monotonic &   $f^8(x) = 4(x-0.5)^2 $        \\
9 & Inverse       &         $f^9(x) = (10 -1/1.1)^{-1} (x+0.1)^{-1} - 0.1 $ \\
		\bottomrule
	\end{tabular}
	}
	\label{tab:MetaF_basis}
\end{table}
\endgroup
\begin{figure}[!h]
	\centering
	\includegraphics[width=14cm]{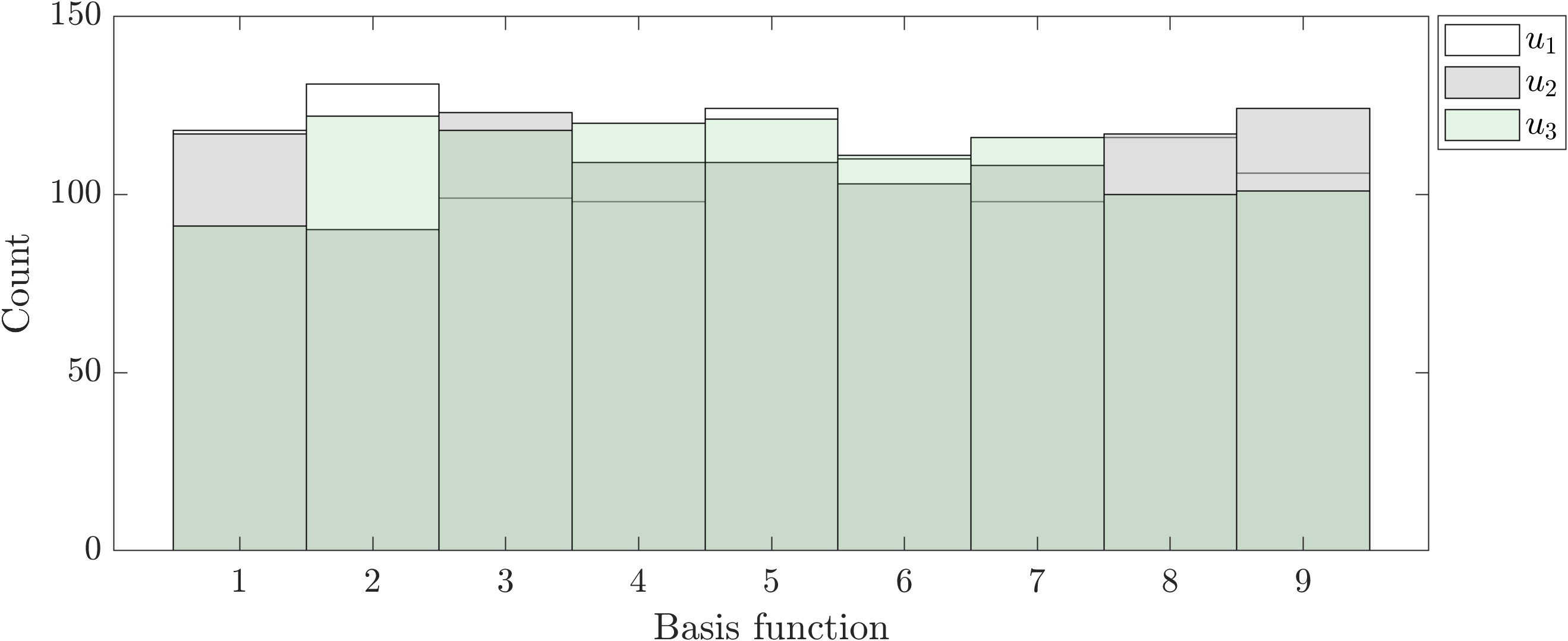}
	\caption{Distribution of basis functions.}
	\label{fig:MetaDist_f}
\end{figure}
\begin{figure}[!h]
	\centering
	\includegraphics[width=14cm]{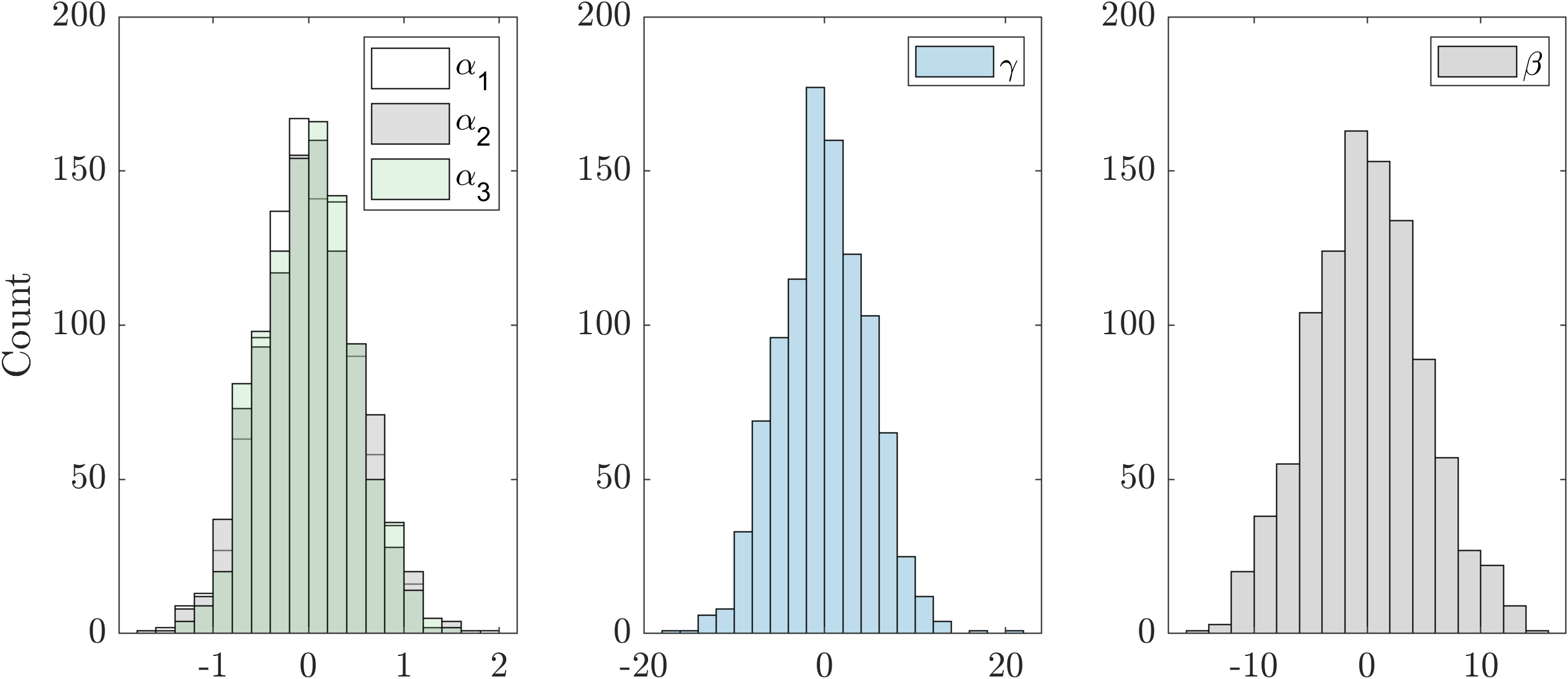}
	\caption{Distribution of coefficients.}
	\label{fig:MetaDist_coef}
\end{figure}

As it is difficult to estimate the conditional entropy for functions of high dimensions, we consider a 3-dimensional random function that is constructed as follows: 
\begin{equation} \label{eq: metaF}   
      y  =  \sum_{i=1}^3 \alpha_i f^{u_i}(x_i) + \beta \prod_{j=1}^2 f^{u_{v_j}} (x_{v_j}) + \gamma \prod_{k=1}^3 f^{u_{w_k}} (x_{w_k}) 
\end{equation}
where $u_i, i = 1,2,3$ is sampled from a discrete uniform distribution over the integers ${1,2, \dots, 9}$. $\alpha_i, \beta, \gamma$ are weighting coefficients that are drawn from a mixture of two standard normal distributions, with variance of 0.5 and 5 respectively. We then get the indices $v_j, j=1,2$ and $w_k, k=1,2,3$ by sampling uniformly from the integer set  ${1,2,3}$, representing the two-factor and three-factor interaction respectively. Note that the same basis function can be drawn multiple times. For example, if $u_i = 1,5,8$, it is possible that $u_{w_k} = 5,5,5$ as they are drawn independently.  

The total effect entropies $H_{T_i}$ are estimated with $10^8$ MC samples for each of these 1000 functions and the results are shown in Figure \ref{fig:MetaF_corr}, in relation to derivative-based $H(X_i) + l_i$. As before, $l_i$ are estimated from 100 samples with a fixed increment step of $10^{-5}$ using finite difference. All three input variables are uniformly distributed, i.e. $x_i \sim \mathbb{U} [0,1], i=1,2,3$. 

The scatter plot from Figure \ref{fig:MetaF_corr} shows that the log-derivative based $l_i$ is well correlated to the total effect entropy $H_{T_i}$, although the upper bound is not satisfied for all functions considered (scatters on the right hand side of the diagonal line indicate $H(X_i) +l_i \leq H_{T_i}$). This is partly due to the slow convergence of $H_{T_i}$ as seen from numerical studies in previous sections, also because the inclusion of discontinuous and dummy basis functions for which the derivative-based upper bounds are not expected. 
\begin{figure}[!h]
	\centering
	 \subfloat[\centering Scatter plot for total effect entropy] {{\includegraphics[width=7cm]{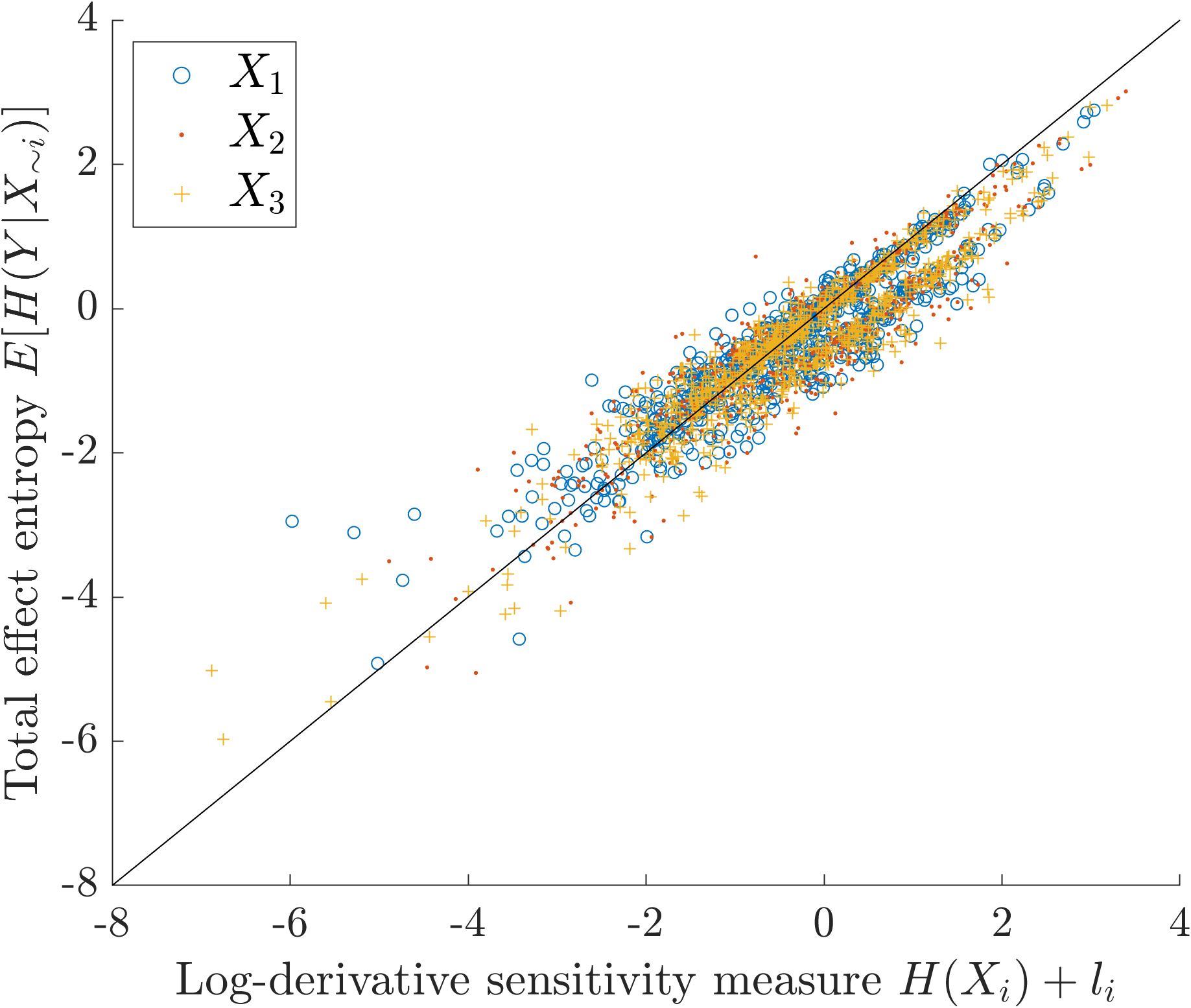} } \label{fig:MetaF_corr}} 
	 \quad
	 \subfloat[\centering Distribution of SA indices $\kappa_{T_i}$ and $l_i$-based upper bound (UB)] {{\includegraphics[width=7cm]{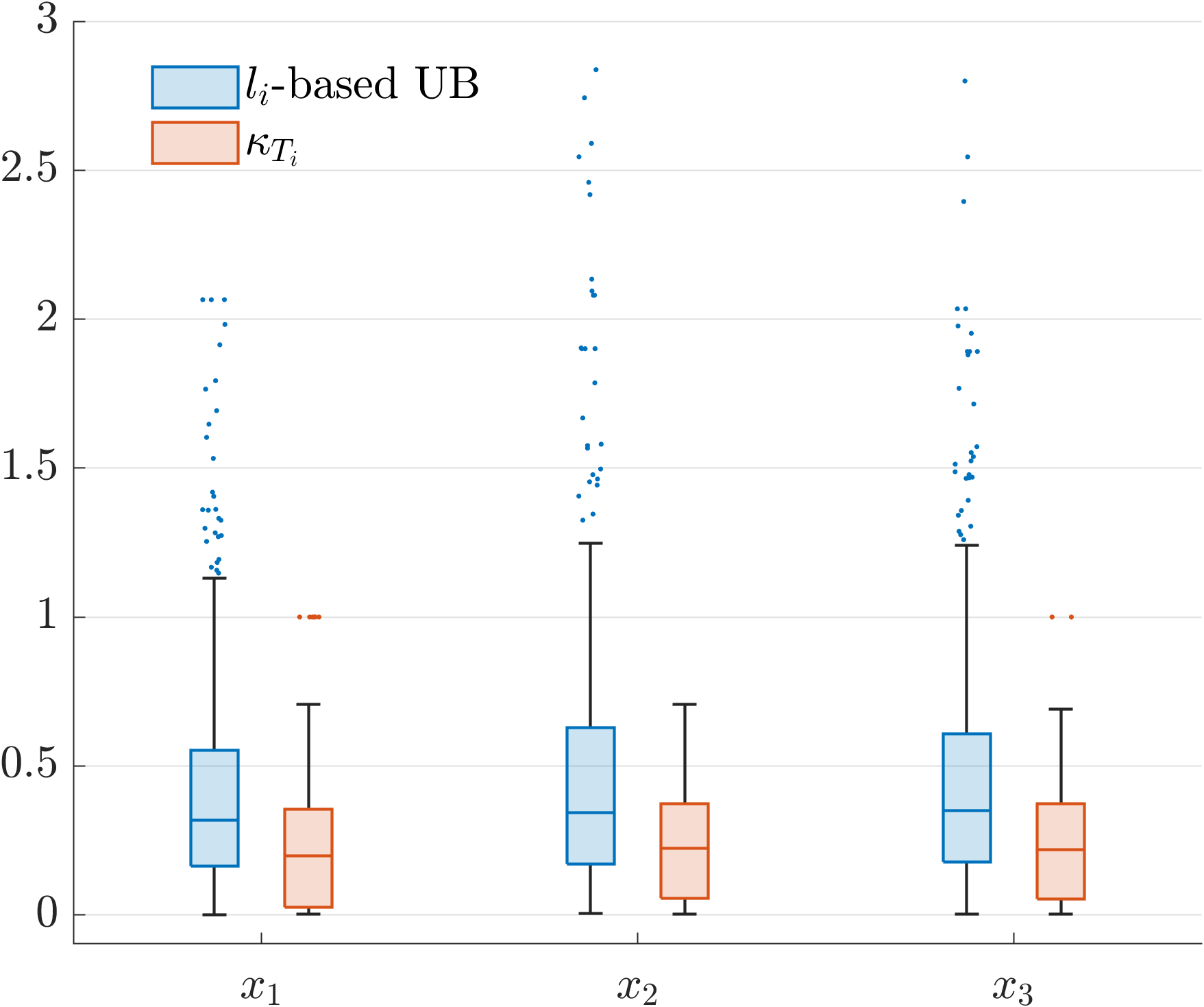} } \label{fig:MetaF_boxplot}} 
	 \caption{Statistical results for sensitivity analysis from 1000 functions drawn randomly from the meta-function}
	\label{fig:MetaF_corr_boxplot}
\end{figure}

A comparison of distributions for eETSI $\kappa_{T_i}$ indices and $l_i$-based upper bound is shown in Figure \ref{fig:MetaF_boxplot}. It can be seen that, excluding the outliers, $l_i$-based upper bounds are reasonably tight for entropy-based indices $\kappa_{T_i}$. This is consistent with the compact scatter we see in Figure \ref{fig:MetaF_corr}. 

The relative tightness of the upper bound suggests its role as a proxy for entropy-based sensitivity analysis. In Figure \ref{fig:MetaF_ranking}, we compare $l_i$-based upper bounds and $\kappa_{T_i}$ in terms of variable ranking from the random functions. The results are categorised as 'All' (same ranking for all three variables), 'max' (most influential variable) and 'min' (least influential variable). 

Out of the 1000 functions, $l_i$-based upper bounds give the exactly same ranking of all variables as $\kappa_{T_i}$ for about three quarters of the functions. The agreements for the the most influential or least influential variables are well above 80\%. Also shown in Figure \ref{fig:MetaF_ranking} are the results from the DGSM $\nu_i$-based UB for entropy-based sensitivity indices, where it can be seen that $l_i$-based UB, which is a tighter bound than $\nu_i$-based UB, agrees better with $\kappa_{T_i}$ for the random functions considered. 
\begin{figure}[!h]
	\centering
	\includegraphics[width=14 cm]{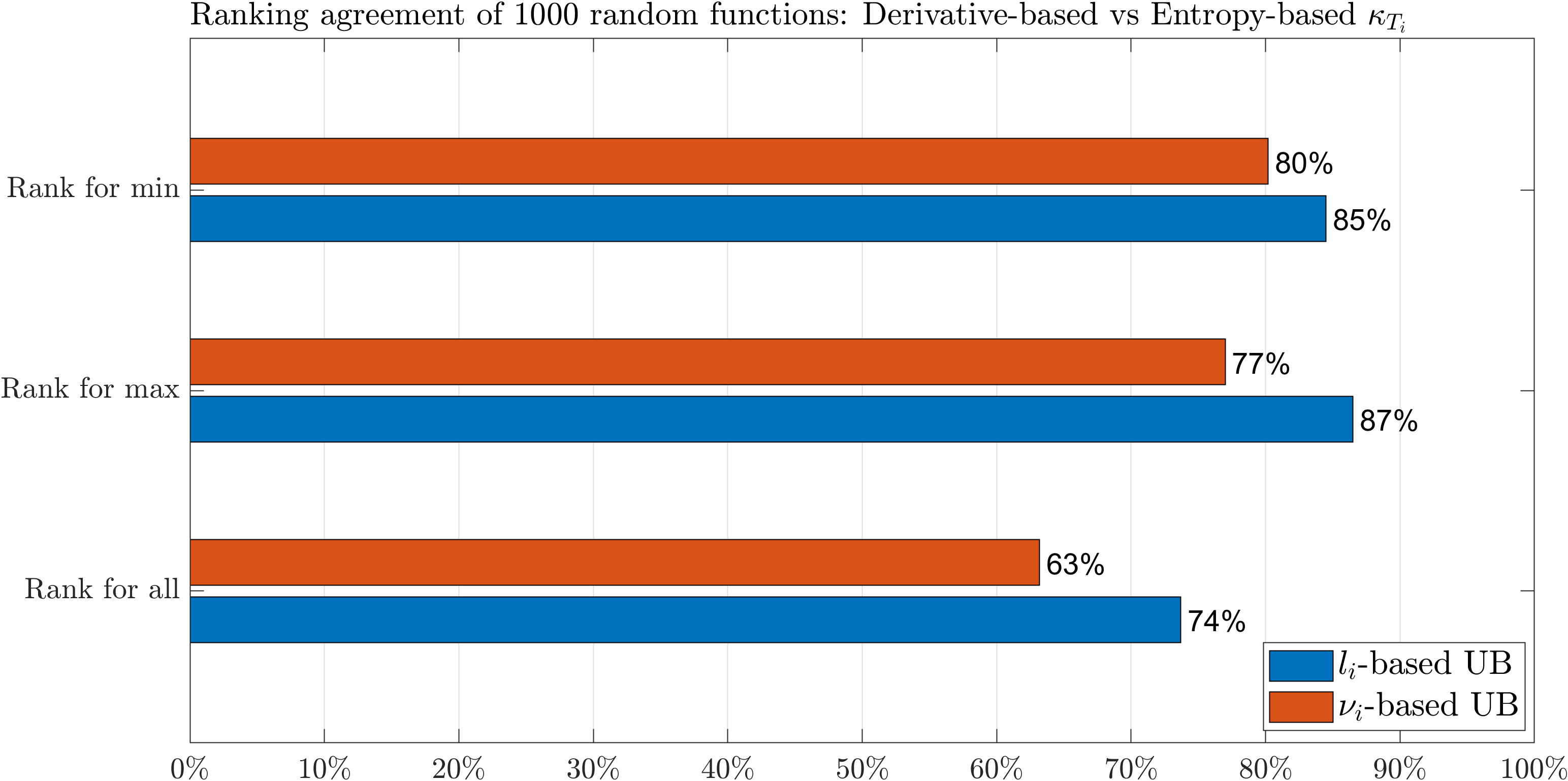}
	\caption{Ranking agreement between derivative-based upper bound (UB) and entropy-based $\kappa_{T_i}$. $l_i$-based UB is $e^{H(X_i)} l_i/e^H(Y)$ and $\nu_i$-based UB is $e^{H(2X_i)} \nu_i/e^H(2Y)$}
	\label{fig:MetaF_ranking}
\end{figure}
\section{A flood model case study}
\label{sec:5}
The numerical examples in the previous section have demonstrated that the log-derivative $l_i$ based upper bound can be potentially used as a screening proxy for entropy-based sensitivity indices, but only limited types of input uncertainty distributions were considered. 

To demonstrate for practical problems with a wide range of input distributions, a simple river flood physics model is considered. This model has been used by \cite{lamboni2013sobol} and \cite{roustant2017poincare} for demonstration of the use of Poincar\'{e} inequality for factor prioritization with DGSM, and as an example in GSA review \citep{iooss2015review}. 

This model simulates the height of a river, and flooding occurs when the river height exceeds the height of a dyke that protects industrial facilities. It is based on simplification of the 1D hydro-dynamical equations of SaintVenant under the assumptions of uniform and constant flow rate and large rectangular sections. The quantity of interest in this case is the maximal annual overflow $Y$:
\begin{equation} \label{eq:flood}   
   Y = Z_{\nu} + D_m - D_d - C_b  \quad  \text{with} \quad   D_m = \left(  \frac{Q}{B K_s \sqrt{(Z_m -Z_{\nu})/L}} \right)^{0.6}
\end{equation}
where the distributions of the independent input variables are listed in Table \ref{tab:floodmodel_input}. We have also added the exponential entropy $e^{H(X_i)}$ value for each variable in Table \ref{tab:floodmodel_input}. The analytical expressions for differential entropy of most probability distributions are readily available and well documented. A comprehensive list can be found in \cite{lazo1978entropy}. 

For a truncated distribution with the interval $[a, b]$, its differential entropy can be found as: \citep{moharana2020truncate}:
\begin{equation} \label{eq:truncatedH}   
   H_{\text{truncated}}(X) = - \int_a^b \frac{f(x)}{\Delta F} \ln \frac{f(x)}{\Delta F} dx
\end{equation}
where $f(x)$ and $F(x)$ are the original probability density function (PDF) and cumulative distribution function (CDF) of the random variable $X$ respectively. ${f(x)}/{\Delta F}$ is the PDF of the truncated distribution where $\Delta F = F(b) - F(a)$. As both $f(x)$ and $F(x)$ are analytically known, Eq \ref{eq:truncatedH} can then integrated numerically for the truncated entropy. 
\begingroup
\def\arraystretch{0.5}
\setlength{\tabcolsep}{4pt}
\begin{table}
	\caption{Entropy and distribution of the flood model input variables}
    \smallskip
    \centering
    \small
\makebox[\textwidth][c] {
	\begin{tabular}{ccllc}
		\toprule
		            & Variable &  Description & Distribution Function & Exponential Entropy $e^{H(X_i)}$ \\  
		\cmidrule(lr){1-5}   
	              $x_1$ & $Q$  & Maximal annual flowrate [m$^3$/s]	 & Truncated Gumbel $\mathcal{G}$(1013,558) on [500,3000]  & 2051 \\
	              $x_2$ & $K_s$  & Strickler coefficient [ - ]	 & Truncated Normal $\mathcal{N}$(30,$8^2$) on [15, $\infty$]  &  30 \\ 
	              $x_3$ & $Z_{\nu}$  & River downstream level [m]	 & Triangular $\mathcal{T}$(49,50,51)   &  1.65 \\ 
	              $x_3$ & $Z_m$  & River upstream level [m]	 & Triangular $\mathcal{T}$(54,55,56)   &  1.65 \\ 
	              $x_5$ & $D_d$  & Dyke height [m]	 & Uniform $\mathcal{U}$[7,9]   &  2 \\ 
	              $x_6$ & $C_b$  & Bank level [m]	 & Triangular $\mathcal{T}$(55,55.5,56)   &  0.825 \\
	              $x_7$ & $L$  & Length of the river stretch [m]	 & Triangular $\mathcal{T}$(4990,5000,5010)   &   16.5\\ 
	              $x_8$ & $B$  & River width [m]	 & Triangular $\mathcal{T}$(295,300,305)   &  8.24 \\ 
		\bottomrule
	\end{tabular}
	}
	\label{tab:floodmodel_input}
\end{table}
\endgroup

The results for the sensitivity indices are shown in Figure \ref{fig:Floodmodel}, where the variance-based $S_{T_i}$ are directly obtained from Table 5 of \cite{lamboni2013sobol}. According to \cite{lamboni2013sobol}, the variance-based index was based on $2 \times 10^7$ model evaluations. The total effect variance indices $S_{T_i}$ have upper bounds that are proportional to the derivative-based DGSM $\nu_i$. The $\nu_i$-variance upper bounds (UB) are also shown in Figure \ref{fig:Floodmodel}, where the optimal Poincar\'{e} constants are used as the proportional constants \citep{roustant2017poincare}.
\begin{figure}[!h]
	\centering
	\includegraphics[width=14 cm]{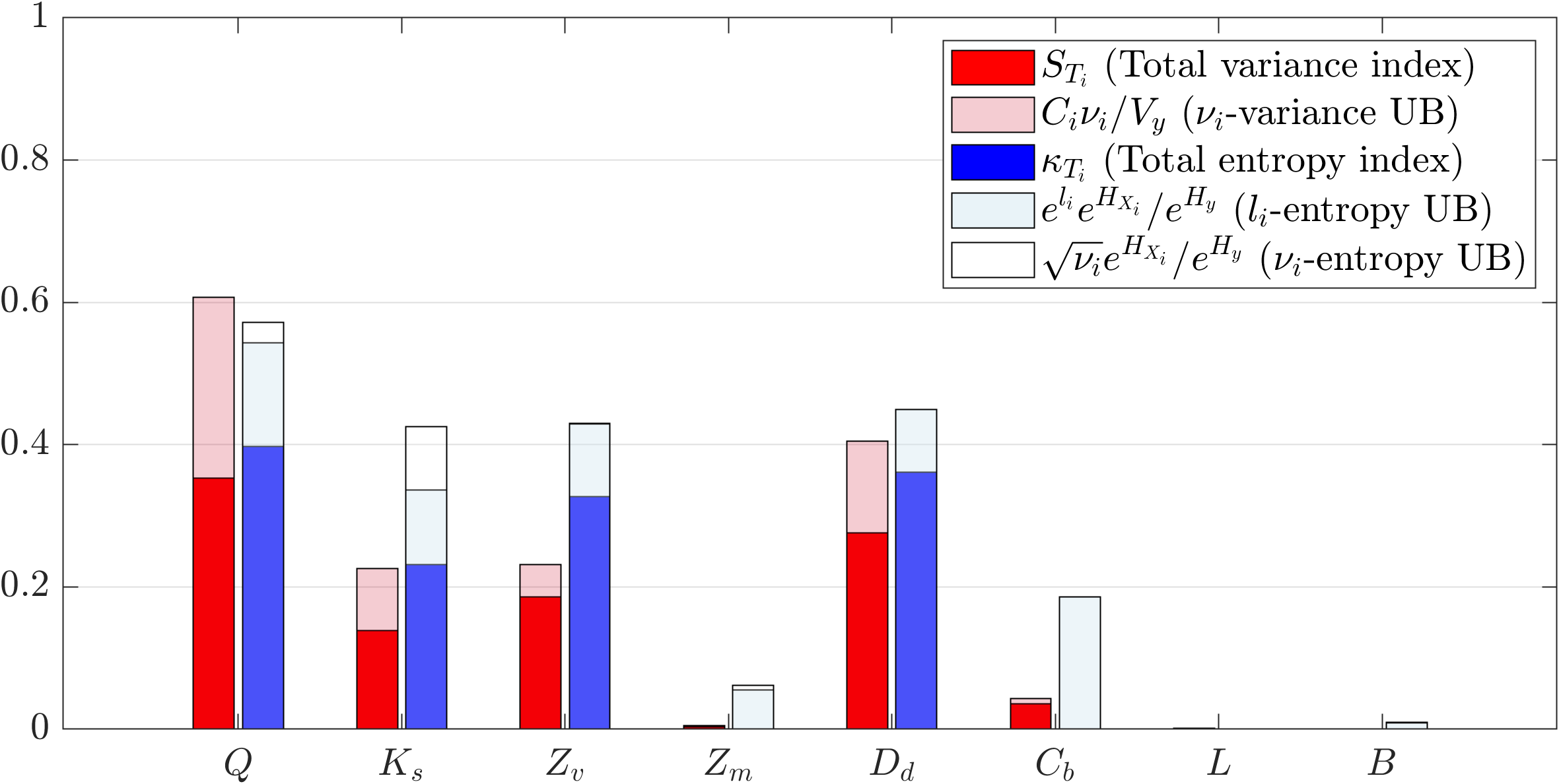}
	\caption{Total effect sensitivity indices for the flood model, from both variance-based $S_{T_i}$ and entropy-based $\kappa_{T_i}$ and their upper bounds (UB). Recall $l_i$ is the log-derivative sensitivity measure and $\nu_i$ is DGSM. Note that for entropy upper bound, we use $\sqrt{\nu_i}$ for a direct comparison with $\kappa_{T_i}$ as seen from Eq \ref{eq:kappa_DGSM}. Numerical data of the results presented here can be found in Appendix \ref{appendix:c}.}
	\label{fig:Floodmodel}
\end{figure}

In comparison, the exponential entropy based total effect sensitivity indices (eETSI) $\kappa_{T_i}$ are also shown in Figure \ref{fig:Floodmodel}, together with the derivative-based upper bounds for eETSI. It can be seen that, similar to the upper bounds for $S_{T_i}$, both $l_i$ and $\nu_i$ based entropy upper bounds are relatively close to $\kappa_{T_i}$, thus providing an efficient proxy for entropy-based total effect sensitivity indices. In this case, we can see that four input variables, $Q, K_s, Z_{\nu},  D_d$, have been identified as the most important variables for maximal annual overflow, with $L$ and $B$ of negligible influence. And this conclusion is consistent from both variance-based $S_{T_i}$ and entropy-based $\kappa_{T_i}$, and their upper bounds. 

Note that as it is computationally difficult to estimate $\kappa_{T_i}$ accurately for this eight dimensional problems, the four least influential variables $Z_m, C_b, L, B$ have been set at their mean values, thus reducing the problem to four dimensions, for the estimation of $\kappa_{T_i}$. Both derivative-based measures have been estimated using all set of input random variables, using the finite difference method with 1000 samples and a fixed increment step of $10^{-5}$. 

From this physics example, we can see that the derivative-based upper bounds can be used as a proxy for the total effect entropy sensitivity analysis for models with a variety of input distributions. Similar to the variance case where the optimal Poincar\'{e} constants can be estimated numerically using the R sensitivity package \citep{da2021book}, it is also straightforward to calculate the normalization constant $e^{H(X_i)}$ for entropy-bounds. Better still, $H(X_i)$ can often be analytically computed as entropy of many distribution functions are known in closed-form. The main computational cost for estimation of the total effect sensitivity proxies is thus the calculation of the partial derivatives, which is much more affordable than a direct estimation of the conditional variance or conditional entropy. 
\section{Conclusions}
\label{sec:6}
A novel global sensitivity proxy for entropy-based total effect has been developed in this paper. We have made use of the inequality between the entropy of the model output and its inputs, which can be seen as an instance of data processing inequality, and established an upper bound for the total effect entropy. This upper bound is tight for monotonic functions. It also provides similar input rankings for about three quarters of the 1000 random functions and can thus be regarded as a proxy for entropy-based total effect measure. Applied to a simplified flood analysis, the new proxy shows good ranking capability for physics problems with a variety of input distributions. 

The resulting log-derivative $l_i$-based proxy is computationally cheap to estimate. If the the derivatives are available, e.g. as the output of a computational code, the proxy would be readily available. Even if a Monte Carlo based approach is used, the computational cost is typically in the order $10^3$ or $10^4$, as compared to $10^7$ or $10^8$ for entropy-based indices. This computational advantage would be even more important for high dimensional problems. However, the numerical assessment has been limited to low dimensional examples, because it becomes prohibitively inefficient to compute conditional entropies with the Monte Carlo based histogram approach adopted in this paper. In subsequent works, we will explore neural network based density estimation techniques for efficient approximation of information-theoretic quantities and extend the numerical assessment of the upper bound to high dimensions.

In many physical applications, the input variables often have a dependence structure due to physical constraints. Unlike DGSM-based proxy, the inequality link between derivatives and entropy presented in this paper does not require the random inputs to be independent. This point is subject to further research, where numerical assessment needs to conducted to explore the screening power of $l_i$-based upper bound for dependent inputs. 

Drawing on the criticism of differential entropy based sensitivity indices, we propose to use its natural exponentiation and the resulting sensitivity measure $\kappa_{T_i}$ possesses many desirable properties for GSA, such as quantitative, moment independent and easy to interpret. The G-function example shows that $\sum \kappa_{T_i}$ is close to one for this product function, as opposed to the variance-based indices where the sum of sensitivity indices is equal to one for additive functions. As exponential entropy can be seen as a geometric mean of the underlying distribution, i.e. $e^{H(X)} = e^{\mathbb{E}[-\ln f(x)]}$, one of the future research is to examine the unique properties of GSA indices based on exponential entropy, and explore its decomposition characteristics for sensitivity analysis of different interaction orders. 

\section*{Acknowledgment}
For the purpose of open access, the author has applied a Creative Commons Attribution (CC BY) licence to any Author Accepted Manuscript version arising. The authors are grateful to Dr Sergei Kucherenko from Imperial College London and Dr Marco De Angelis from University of Strathclyde for comments on an early draft of this article, and Dr Bertrand Iooss from EDF and Dr Eric Hall from University of Dundee for useful discussions and reviewing of this work. 
\section*{Data availability statement}
The authors confirm that the data supporting the findings of this study are available within the article.

\bibliographystyle{unsrtdin}
\bibliography{references}  

\begin{appendices}
\section{Numerical estimation of entropy}
\label{appendix:a}

\renewcommand{\theequation}{\thesection.\arabic{equation}}
\setcounter{equation}{0}

Adopting the approach from \cite{moddemeijer1989estimation}, the $xy$-plane is gridded by equal size cells ($\Delta x \times \Delta y$) with coordinates ($i$, $j$). The probability of observing a sample in cell ($i$, $j$) is: 
\begin{equation} \label{eq:Est.1}   
      p_{ij} = \iint\limits_{\text{cell} (i,j)} f(x,y)dxdy \approx f(x_i,y_j)\Delta x \Delta y
\end{equation}
where $(x_i,y_j)$ is the centre of the cell. 

Assuming the jPDF is approximately constant within a cell, the joint entropy can be represented as:
\begin{equation} \label{eq:Est.2}
   \begin{split}
       H(X,Y) &  =  - \int  f(x,y) \ln f(x,y) dx dy  \\
                    &  \approx  - \sum f(x_i,y_j) \ln f(x_i,y_j) \Delta x \Delta y \\
                    &  \approx  - \sum p_{ij} \left( \ln p_{ij} - \ln (\Delta x \Delta y) \right) \\
                  &  \approx  - \sum \left( \frac{k_{ij}}{N} \ln \frac{k_{ij}}{N} \right) + \ln (\Delta x \Delta y)
  \end{split} 
\end{equation}
where $k_{ij}$ represents the number of samples observed in the cell ($i$, $j$), and $N$ is the total number of samples. 

Similarly, the conditional entropy can be approximated as:
\begin{equation} \label{eq:Est.3}
          \begin{split}
       H(Y|X) &  =  - \int  f(x,y) \frac{\ln f(x,y)}{f(x)} dx dy  \\
                    &  \approx  - \sum p_{ij} \left( \ln p_{ij} - \ln p_i-\ln \Delta y \right) \\
                  &  \approx  - \sum \left( \frac{k_{ij}}{N} \ln \frac{k_{ij}}{k_i} \right) + \ln \Delta y
  \end{split} 
\end{equation}
where $k_i = \sum_j k_{ij}$ and similar expressions can be derived when $\mathbf X$ is a vector variable. 

\section{Analytical derivations for monotonic examples }
\label{appendix:b}
We present in this section the derivation of the total entropy inequality from Eq \ref{eq:proposition} for the five monotonic functions. All input variables are assumed to have the same uniform distribution for examples 1 - 4, i.e. $X_i \sim \mathbb{U}(0,1)$, while Gaussian distributions are used for example 5. For verification purposes, all the examples in this section are chosen to have tractable expressions for both the integral of derivatives and the conditional entropies. 

\paragraph*{Example 1.}  $y=x_1 + e^{x_2}$ \\ 
For variable $X_1$, $H_{T_1} =\mathbb{E}_{X_2}\left[H(Y | X_2 ) \right] = \mathbb{E}_{X_2}\left[H(X_1 ) \right] = H(X_1) = 0$ because the differential entropy remains constant under addition of a constant ($e^{x_2}$ is a constant for $H(Y|X_2 )$). For right hand side of the inequality in Eq \ref{eq:proposition}, we have $g_{x_1} = \partial y/\partial x_1 = 1$, so $\mathbb E \left[ \ln {\left| g_{x_1} \right|} \right] + H(X_1) = 0 + 0 = 0$. \\
For variable $X_2$, $H_{T_2} = \mathbb{E}_{X_1}\left[H(Y | X_1 ) \right]  =\int_{X_1} H(Y|X_1 = x_1) dx_1 = \int_{X_1} \left[ - \int_1^e \frac{1}{y}\ln\frac{1}{y}  dy \right] dx_1 = 1/2$,  where $1/y$ is the conditional PDF of $Y|X_1$ because the transformed variable $V=e^{X}$ has a PDF: $V \sim 1/v, 1 < v < e$. Similarly, we have $\partial y/\partial x_2 = e^{x_2}$, so $\mathbb E \left[ \ln {\left| g_{x_2} \right|} \right] - H(X_2) = 1/2 - 0 = 1/2$, which proves the equality as expected. \\
\paragraph*{Example 2.}  $y=x_1 \times x_2$\\
Consider the function $y=x_1 \times x_2$, where the partial derivatives are $\partial y/\partial x_1 = x_2$ and $\partial y/\partial x_2 = x_1$. For $X_i \sim \mathbb{U}(0,1)$, the expected value  $l_i = \mathbb E \left[ \ln {\left| g_{x_i} \right|} \right]$ are $-1$ for both $x_1$ and $x_2$. Recall that the differential entropy increases additively upon multiplication with a constant, $H_{T_1}   = \mathbb{E}_{X_2}\left[ H(X_1) + \ln|X_2 = x_2| \right] = \int_0^1 \ln{x_2} d{x_2} = -1$. Same results can be obtained for $X_2$ due to the symmetry between $X_1$ and $X_2$. Therefore, $H_{T_i}=\mathbb{E}[H(Y|\mathbf{X}_{\sim i}) ]= H(X_i) + l_i$ as given by Eq \ref{eq:proposition}. 
\paragraph*{Example 3.}  $y=x_1 + 3x_2$\\
Consider the function $y=x_1 + 3x_2$, where the partial derivatives are 1 and 3  for $x_1$ and $x_2$ respectively. For $X_i \sim \mathbb{U}(0,1)$, the expected value  $l_i = \mathbb E \left[ \ln {\left| g_{x_i} \right|} \right]$ can be integrated analytically as 0 and $\ln3 \simeq 1.0986$ respectively. It is straightforward to show that, as in previous examples, $H_{T_1} = \mathbb E \left[ \ln {\left| g_{x_1} \right|} \right] - H(X_1) = 0$ and $H_{T_2}  = \ln3$, which is the same as $l_2 + H(X_2)$ . 
\paragraph*{Example 4.}  $y= x_1x_2^r$\\
Consider the product function $y= x_1x_2^r$, where $r \geq 1$ (for $r=1$, we recover the function in example 2).  For $X_1, X_2 \sim \mathbb{U}(0,1)$, it is straightforward to show that $\mathbb E \left[ \ln {\left| g_{x_1} \right|} \right]  = - r$ and $\mathbb E \left[ \ln {\left| g_{x_2} \right|} \right] = \ln{r} -r$. \\
For variable $X_1$, $H_{T_1} = \mathbb{E}_{X_2}\left[ H(X_1) + \ln|x_2| \right] = \int_0^1 \ln{x_2^r} d{x_2} = -r$, which is the same as $l_1+ H(X_1) = -r + 0 = -r$.  \\
For variable $X_2$, $H_{T_2} = \mathbb{E}_{X_1}\left[ H(Y|X_1) + \ln|x_1|  \right ] =  \mathbb{E}_{X_1}\left[ - \int_0^1 p(y|x_1)\ln p(y|x_1) \right]  - 1 = \ln{r} - r $, where the conditional PDF of $Y|X_1 = x_1$ is $p(y|x_1) =\frac{1}{r}y^{\frac{1}{r}-1} $ because the transformed variable $V=x^{r}$ has a PDF: $V \sim \frac{1}{r}v^{\frac{1}{r}-1}, 0 < v < 1$. This is the same as $l_2 + H(X_2) = \ln{r} -r - 0 = \ln{r} -r$.
\paragraph*{Example 5.}  $y= \sum_{i=1}^n a_i x_i$\\
This linear function, $y= \sum_{i=1}^n a_i x_i$, has been used in \cite{krzykacz2001epistemic} to demonstrate the equivalence between entropy based and variance based sensitivity indices for Guassian random inputs, i.e. $X_i \sim \mathbb{N}(\mu_i,\sigma_i^2)$. In the case with independent inputs, the sensitivity index based on the conditional entropy can be obtained as $ H_{T_i}  =1/2 \ln(2 \pi e a_i^2 \sigma_i^2) $, and this is just the logarithmically scaled version of $H(X_i)$, i.e. $ H_{T_i}= H(X_i) + \ln {|a_i|}$. As $\mathbb E \left[ \ln {\left| g_{x_i} \right|} \right]= \ln|a_i|$ for this simple linear function, the special equality case is obtained for the inequality given in Eq \ref{eq:proposition}. 

\section{Numerical data for flood model}
\label{appendix:c}
\renewcommand{\thetable}{\Alph{section}\arabic{table}}
\setcounter{table}{0}
The numerical data for Figure \ref{fig:Floodmodel} are listed in Table \ref{tab:floodmodel_data}, where the intermediate results of $e^{l_i}$, $\nu_i$ and the Poincar\'{e} optimal constants are listed in Table \ref{tab:floodmodel_data2}. Based on Table \ref{tab:floodmodel_data}, the variable ranking results are shown in Table \ref{tab:floodmodel_ranking}. 
\begingroup
\def\arraystretch{0.5}
\setlength{\tabcolsep}{6pt}
\begin{table}[ht]
	\caption{Sensitivity results for the flood model.}
    \smallskip
    \centering
\makebox[\textwidth][c] {
	\begin{tabular}{ccccccc} 	
		\toprule
		            & Variable &  $S_{T_i}$	& $\nu_i$-variance 	& $\kappa_{T_i}$	 & $l_i$-entropy & 	$\nu_i$-entropy \\
		\cmidrule(lr){1-7}   
$x_1$ & $Q$       & 0.353 & 0.607 & 0.397 & 0.543 & 0.572 \\
$x_2$ & $K_s$     & 0.139 & 0.226 & 0.231 & 0.336 & 0.425 \\
$x_3$ & $Z_{\nu}$ & 0.186 & 0.232 & 0.327 & 0.429 & 0.430 \\
$x_3$ & $Z_m$     & 0.003 & 0.005 &   -    & 0.055 & 0.061 \\
$x_5$ & $D_d$     & 0.276 & 0.405 & 0.361 & 0.450 & 0.450 \\
$x_6$ & $C_b$     & 0.036 & 0.043 &    -   & 0.186 & 0.186 \\
$x_7$ & $L$       & 0.000 & 0.000 &    -   & 0.001 & 0.001 \\
$x_8$ & $B$       & 0.000 & 0.000 &   -    & 0.009 & 0.010\\
		\bottomrule
	\end{tabular}
	}
	\label{tab:floodmodel_data}
\end{table}
\endgroup
\begingroup
\def\arraystretch{0.5}
\setlength{\tabcolsep}{6pt}
\begin{table}[ht]
	\caption{Data for the derivative-based measures and Poincar\'{e} optimal constant}
    \smallskip
    \centering
\makebox[\textwidth][c] {
	\begin{tabular}{ccccc} 	
		\toprule
		            & Variable &  $e^{l_i}$	& $\nu_i$ 	&  Poincar\'{e} optimal constant \\
		\cmidrule(lr){1-5}   
$x_1$ & $Q$       & 0.001 & 0.000 & 3.93E+05 \\
$x_2$ & $K_s$     & 0.050 & 0.004 & 5.77E+01 \\
$x_3$ & $Z_{\nu}$ & 1.155 & 1.339 & 1.73E-01 \\
$x_3$ & $Z_m$     & 0.148 & 0.027 & 1.73E-01 \\
$x_5$ & $D_d$     & 1.000 & 1.000 & 4.05E-01 \\
$x_6$ & $C_b$     & 1.000 & 1.000 & 4.32E-02 \\
$x_7$ & $L$       & 0.000 & 0.000 & 1.73E+01 \\
$x_8$ & $B$       & 0.005 & 0.000 & 4.32E+00\\
		\bottomrule
	\end{tabular}
	}
	\label{tab:floodmodel_data2}
\end{table}
\endgroup
\begingroup
\def\arraystretch{0.5}
\setlength{\tabcolsep}{6pt}
\begin{table}[ht]
	\caption{Variable ranking for the flood model.}
    \smallskip
    \centering
\makebox[\textwidth][c] {
	\begin{tabular}{ccccccc} 	
		\toprule
		            & Variable &  $S_{T_i}$	& $\nu_i$-variance 	& $\kappa_{T_i}$	 & $l_i$-entropy & 	$\nu_i$-entropy \\
		\cmidrule(lr){1-7}   
$x_1$ & $Q$       & 1 & 1 & 1 & 1 & 1 \\
$x_2$ & $K_s$     & 4 & 4 & 4 & 4 & 4 \\
$x_3$ & $Z_{\nu}$ & 3 & 3 & 3 & 3 & 3 \\
$x_3$ & $Z_m$     & 6 & 6 & -  & 6 & 6 \\
$x_5$ & $D_d$     & 2 & 2 & 2 & 2 & 2 \\
$x_6$ & $C_b$     & 5 & 5 & -  & 5 & 5 \\
$x_7$ & $L$       & 7 & 7 & -  & 8 & 8 \\
$x_8$ & $B$       & 8 & 8 &  - & 7 & 7\\
		\bottomrule
	\end{tabular}
	}
	\label{tab:floodmodel_ranking}
\end{table}
\endgroup

\section{Interpretation of sensitivity results for Ishigami function}
\label{appendix:d}
\renewcommand{\thetable}{\Alph{section}\arabic{table}}
\setcounter{table}{0}
\renewcommand{\thefigure}{\Alph{section}.\arabic{figure}}
\setcounter{figure}{0}
For the Ishigami function results presented in Section 4.2 of the \textit{Main Text}, both indices have successfully identified the contribution of $X_3$ which is the lowest. However, the relative importance of $X_1$ and $X_2$ are opposite from $\kappa_i$ and $S_{T_i}$.  This difference between variance-based and entropy-based ranking for Ishigami function was also noted in \cite{auder2008entropy}, but without further explanation. 

We attempt to look at the difference using a simple function with $(x, z)$ as the input random variables:
\begin{equation} \label{eq:y=ax+b}   
      y(x, z) =  \alpha (z) x + \beta(z)
\end{equation}
which can be seen to represent the interaction effect of $x_1$ and $x_3$ in the Ishigami function. The conditional variance and entropy of $Y| Z=z$ are thus:
\begin{equation} \label{eq: y|z}   
      V(Y|Z=z) =  {\alpha^2 (z)} V(X)   \quad  \text{and}  \quad  H(Y|Z=z) = H(X)  + \ln {|\alpha (z)| }   
\end{equation}
where the translation invariance of variance and entropy are used. And the corresponding un-conditional $\mathbb{E}[V(Y|Z)]$ and $ \mathbb{E}[H(Y|Z)] $ can be found as:
\begin{equation} \label{eq: E(y|z)}   
      \mathbb{E}[V(Y|Z)] =   \mathbb{E}_Z [{\alpha^2 (z)}] V(X)   \quad  \text{and}  \quad   \mathbb{E}[H(Y|Z)] =  \mathbb{E}_Z [\ln {|\alpha (z)|} ] + H(X)
\end{equation}
From Eq \ref{eq: E(y|z)}, it is clear that despite the similarities, the expected variance $\mathbb{E}[V(Y|Z)]$ depends on $\alpha^2$ which grows much faster than $\ln|\alpha|$ for the conditional entropy $ \mathbb{E}[H(Y|Z)]  $ as $|\alpha|$ increases. 

In the Ishigami function, the effect of $x_3$ grows fast as $x_3$ approach its boundary, i.e. towards $-\pi$ and $\pi$. The much stronger interaction at the support boundary between $x_1$ and $x_3$ has also been noted in \cite{fruth2019support}, where the importance of $x_1$ was found to be much reduced if the support of the distribution of $x_3$ is reduced by $10$ percent. Putting $x_3 \sim [-\pi + \pi/10, \pi - \pi/10]$ would put $x_2$ as the most influential variable. 

From Eq \ref{eq: E(y|z)}, we can see that the interaction between $x_1$ and $x_3$ is more influential for variance due to the squaring effect, as compared to the entropy operation which takes logarithm of the interaction. This difference increases towards the boundary as the interaction between $x_1$ and $x_3$ gets stronger towards $-\pi$ and $\pi$. And this helps to explain why $x_1$ is the most influential variable for the variance-based $S_{T_i}$. This example highlights that, despite many similarities, entropy and variance are fundamentally different, for example, the variable interactions are processed differently between them.

\section{Illustration with group input variables}
\label{appendix:e}
\renewcommand{\thetable}{\Alph{section}\arabic{table}}
\setcounter{table}{0}
\renewcommand{\thefigure}{\Alph{section}.\arabic{figure}}
\setcounter{figure}{0}
The log-derivative $l_i$ based upper bound (UB) can be adjusted to work with groups of input variables. Consider an arbitrary subset of the input variables $z = (x_{i1}, \dots, x_{is}), 1 \le s < d $, $ \mathbb{E} \left[ \ln {\left| {\partial g (\mathbf{x})}/{\partial z} \right|} \right]$ can then be used for $l_i$ to assess the effect of the group $z$.  ${\partial g}/{\partial z} $ can be calculated using the total derivative of the function $g$, by fixing the remaining complementary set of $z$. 

We assess the application of $l_i$-based UB to the G-function. We follow \cite{campolongo2007effective} and consider a G-function with 9 input factors, all uniformly
distributed in the range [0,1], and build three test cases where groups have different features, as listed in Table \ref{tab:GroupCase}. Note that the values of $a_i$ are different from the case considered in the main text. 

More precisely, we estimate the exponentiated UB, i.e. ${e^{H(Z)}}/{e^{H(Y)}} e^{l_i}$. In this case, $e^{H(Z)} = 1$ for different input groups as they have the same distributions and are assumed to be independently sampled. In addition, $e^{H(Y)}$ are $2.625, 1.950, 2.625 $ for the three cases considered. 

The sensitivity results are given in Table \ref{tab:GroupCase_SA}. As compared with the Morris-based $\mu*$ and variance-based total effect $S_T$, the results confirm the fitness of the $l_i$-based UB to handle groups of factors for the G-function tested. A similar level of importance among the three groups has been successfully identified by the three methods. In the more complex case 3, where each of the three groups contains a mix of important and unimportant variables, the relative group importance from Morris-based $\mu*$ is different from $S_T$, while $l_i$-based UB gives exactly the same ranking as $S_T$. 
\begingroup
\def\arraystretch{0.5}
\setlength{\tabcolsep}{6pt}
\begin{table}[ht]
	\caption{Group effect assessment: parameters for the G-function}
    \smallskip
    \centering
\makebox[\textwidth][c] {
\begin{tabular}[c]{cccccccccc}
		\toprule
		& \multicolumn{9}{c}{$y = \prod_{i=1}^{9} (|4x_i - 2| + a_i )/(1+a_i)$} \\
		\\
       & \multicolumn{3}{c}{Group 1} & \multicolumn{3}{c}{Group 2} & \multicolumn{3}{c}{Group 3} \\ 
        \cmidrule(lr){2-4}   \cmidrule(lr){5-7}  \cmidrule(lr){8-10} 
       & $a_1$   & $a_2$   & $a_3$   & $a_4$   & $a_5$   & $a_6$   & $a_7$   & $a_8$   & $a_9$   \\
        \cmidrule(lr){2-10} 
Case 1 & 0.02    & 0.03    & 0.05    & 11      & 12.5    & 13      & 34      & 35      & 37      \\
           \cmidrule(lr){2-10} 
Case 2 & 0.02    & 0.04    & 0.06    & 0.03    & 0.05    & 0.07    & 34      & 35      & 37      \\
         \cmidrule(lr){2-10} 
Case 3 & 0.02    & 11      & 35      & 0.05    & 12.5    & 37      & 0.03    & 13      & 14     \\
		\bottomrule
	\end{tabular}
	}
	\label{tab:GroupCase}
\end{table}
\endgroup
\begingroup
\def\arraystretch{1}
\setlength{\tabcolsep}{6pt}
\begin{table}[ht]
	\caption{Group effect assessment: sensitivity results from $l_i$-based UB, with $\mu*$  and $S_T$ given by \cite{campolongo2007effective}}
    \smallskip
    \centering
\makebox[\textwidth][c] {
\begin{tabular}[c]{cccccccccc}
		\toprule
       & \multicolumn{3}{c}{Morris-based $\mu*$} & \multicolumn{3}{c}{Variance-based $S_T$} & \multicolumn{3}{c}{$l_i$-based UB} \\ 
        \cmidrule(lr){2-4}   \cmidrule(lr){5-7}  \cmidrule(lr){8-10} 
Group No. & 1      & 2      & 3     & 1      & 2     & 3     & 1        & 2       & 3       \\
Case 1    & 7.948  & 1.058  & 0.708 & 0.995  & 0.010 & 0.001 & 4.425    & 0.349   & 0.126   \\
Case 2    & 42.339 & 32.656 & 2.735 & 0.694  & 0.686 & 0.001 & 5.917    & 5.861   & 0.170   \\
Case 3    & 8.108  & 7.083  & 6.364 & 0.436  & 0.393 & 0.429 & 1.664    & 1.604   & 1.632  \\
		\bottomrule
	\end{tabular}
	}
	\label{tab:GroupCase_SA}
\end{table}
\endgroup

\end{appendices}

\end{document}